\PassOptionsToPackage{dvipsnames,svgnames,table}{xcolor}

\documentclass[11pt]{amsart}  

\usepackage[vmargin=1in, hmargin=1.3in]{geometry}  
\usepackage[font=small,format=plain,labelfont=bf,up,textfont=it,up]{caption}
\usepackage{microtype}                  		
\usepackage{graphicx}						
\usepackage{amssymb}
\usepackage{amsmath}
\usepackage{amsthm}
\usepackage{xcolor}
\usepackage{wasysym}
\usepackage{array}
\usepackage{pst-node}
\usepackage{tikz-cd}
\usepackage{tikz}
\usepackage{tikz-qtree}
\usepackage{mathtools}
\usepackage{amsfonts}
\usepackage{bbm}
\usepackage{listings}
\usepackage{MnSymbol}
\usepackage{multirow}
\usepackage{centernot}
\usepackage{ stmaryrd }
\usepackage{subcaption}
\usepackage[backref=page, bookmarks, bookmarksdepth=2, colorlinks=true, linkcolor=blue, citecolor=blue, urlcolor=blue]{hyperref}
\usepackage{enumitem}

\usepackage[export]{adjustbox}
\usepackage{verbatim}
\usepackage{makecell}
\usepackage{xfrac}
\usepackage{multicol}
\usepackage{float}
\usepackage{fancyhdr}

\apptocmd{\thebibliography}{\raggedright}{}{}
\renewcommand*{\backref}[1]{}
\renewcommand*{\backrefalt}[4]{%
    \ifcase #1 (Not cited.)%
    \or        (Cited on page~#2.)%
    \else      (Cited on pages~#2.)%
    \fi}

\theoremstyle{plain}
\newtheorem*{theorem*}{Theorem}
\newtheorem*{conjecture*}{Conjecture}
\newtheorem*{corollary*}{Corollary}
\newtheorem{theorem}{Theorem}[section]
\newtheorem{maintheorem}{Theorem}

\newtheorem{proposition}[theorem]{Proposition}
\newtheorem{lemma}[theorem]{Lemma}

\newtheorem{corollary}[theorem]{Corollary}
\newtheorem{conjecture}[theorem]{Conjecture}

\theoremstyle{definition}

\theoremstyle{remark}

\theoremstyle{remark}
\newtheorem{rmk}[theorem]{Remark}
\newenvironment{remark}[1][]{\begin{rmk}[#1]}{\end{rmk}}
\newtheorem{eg}[theorem]{Example}

\definecolor{colorblind_blue}{RGB}{0,114,178}
\definecolor{colorblind_orange}{RGB}{213,94,0}
\definecolor{colorblind_green}{RGB}{0,158,115}
\definecolor{colorblind_purple}{RGB}{204,121,167}
\definecolor{colorblind_darkpurple}{RGB}{126,41,84}

\newcommand{\orange}[1]{\ensuremath{{\color{colorblind_orange} #1}}}
\renewcommand{\blue}[1]{\ensuremath{{\color{colorblind_blue} #1}}}
\renewcommand{\green}[1]{\ensuremath{{\color{colorblind_green} #1}}}

\newcommand{\arxiv}[1]{\href{http://arxiv.org/abs/#1}{{\tt arXiv:#1}}}

\newcommand{\Z}{\mathbb Z}
\newcommand{\Q}{\mathbb Q}

\newcommand{\F}{\mathbb F}

\newcommand{\ai}{\bar{a}_i}
\newcommand{\aj}{\bar{a}_j}
\newcommand{\ak}{\bar{a}_k}
\newcommand{\al}{\bar{a}_{\ell}}
\newcommand{\am}{\bar{a}_m}

\newcommand{\bi}{\bar{b}_i}
\newcommand{\bj}{\bar{b}_j}
\newcommand{\bk}{\bar{b}_k}
\newcommand{\bl}{\bar{b}_{\ell}}
\newcommand{\bm}{\bar{b}_m}
\newcommand{\bn}{\bar{b}_n}

\newcommand{\llsp}{\ell \ell Sp_{2g}(\F_2)}

\allowdisplaybreaks

\newcommand\Mod{\ensuremath{\operatorname{Mod}}}

\title[The Birman--Craggs--Johnson homomorphism and the handlebody Torelli group]{The Birman--Craggs--Johnson homomorphism and the handlebody Torelli group}

\author{Annie Holden}
\address{Department of Mathematics, University of Pennsylvania, Philadelphia, PA, USA}
\email{holdena@sas.upenn.edu}

\begin{document}

\newpage

\begin{abstract}

We use the Birman--Craggs--Johnson (BCJ) homomorphism to study the intersection of the handlebody group with the Torelli group and with the Johnson kernel. For genus $\geq 3$, we determine the images of the BCJ homomorphism restricted to these subgroups. We then explicitly compute cup products of pairs of cohomology classes in $H^1(-;\F_2)$ detected by the BCJ homomorphism for genus $\geq 4$. For the handlebody Johnson kernel, our results lift to integral coefficients.

\end{abstract}

\pagestyle{headings}
\maketitle
\thispagestyle{empty}

\section{Introduction} \label{section:intro}

\subsection{Surface mapping class groups} Let $\Sigma_g^1$ be an oriented surface of genus $g$ with one embedded disk. The mapping class group $\Mod_g^1$ is the group of isotopy classes of orientation-preserving diffeomorphisms of $\Sigma_g^1$ that fix the embedded disk pointwise. The action of $\Mod_g^1$ on $H_1(\Sigma_g^1; \Z)$ preserves the symplectic form. The kernel of this action is the \emph{Torelli group} $\mathcal{I}_g^1$, yielding the short exact sequence $$1 \rightarrow \mathcal{I}_g^1 \rightarrow \Mod_g^1 \rightarrow \operatorname{Sp}_{2g}(\Z) \rightarrow 1.$$ Understanding the cohomology $H^*(\mathcal{I}_g^1)$ is a central problem in geometric topology. In degrees one and two, rational cohomology is detected by the Johnson homomorphism (see \cite{JohnsonIII} and \cite{MinahanPutman}), whose kernel $\mathcal{K}_g^1$ (the \emph{Johnson kernel}) has deep connections to Heegaard splittings and the Casson invariant. To capture $2$-torsion, Johnson \cite{BCJ} defined the Birman--Craggs--Johnson (BCJ) homomorphism $$\sigma: \mathcal{I}_g^1 \rightarrow B_3$$ into an $\F_2$-vector space $B_3$, which induces an isomorphism $H_1(\mathcal{I}_g^1; \F_2) \cong B_3$ and connects $\mathcal{I}_g^1$ to the Rochlin invariant of homology $3$-spheres (see \cite{EellsKuiper}). Johnson also showed that the restriction of the BCJ homomorphism to $\mathcal{K}_g^1$ surjects onto the subspace $B_2 \subset B_3$, which is identified with the $2$-torsion subgroup of $H_1(\mathcal{I}_g^1; \Z)$. By analyzing this restriction map, Brendle--Farb \cite{BrendleFarb} constructed elements of $H^2(\mathcal{I}_g^1; \F_2)$ that cannot be detected rationally.

\subsection{Handlebody groups} \label{section:introhandlebody} Now fix a handlebody $\mathcal{V}_g^1$ with one embedded disk in its boundary so that $\partial \mathcal{V}_g^1 = \Sigma_g^1$. The handlebody group $\mathcal{H}_g^1$ is the subgroup of $\Mod_g^1$ that extends to $\mathcal{V}_g^1$. We define the \emph{handlebody Torelli group} by $\mathcal{HI}_g^1 := \mathcal{H}_g^1 \cap \mathcal{I}_g^1$ and the \emph{handlebody Johnson kernel} by $\mathcal{HK}_g^1 := \mathcal{H}_g^1 \cap \mathcal{K}_g^1$. These subgroups play a key role in the topology of $3$-manifolds (see, e.g., \cite[Section 3]{Hensel}).

\subsection{Structure and main results} This paper is the second in a series dedicated to studying $\mathcal{HI}_g^1$ and $\mathcal{HK}_g^1$ by adapting techniques from the surface mapping class group setting. In \cite[Theorem A]{Holden}, we investigated $H^2(\mathcal{HI}_g^1; \Q)$ via the Johnson homomorphism. Our methods did not yield any information about torsion or $\mathcal{HK}_g^1$. The natural next step is therefore to study $H^*(\mathcal{HI}_g^1; \F_2)$ and $H^*(\mathcal{HK}_g^1; \F_2)$ using the BCJ homomorphism. In our first result, we compute the images $\sigma(\mathcal{HI}_g^1)$ and $\sigma(\mathcal{HK}_g^1)$ of the BCJ homomorphism restricted to these subgroups. This identifies $H^1(\sigma(\mathcal{HI}_g^1); \F_2)$ and $H^1(\sigma(\mathcal{HK}_g^1); \F_2)$ with subspaces of $H^1(\mathcal{HI}_g^1; \F_2)$ and $H^1(\mathcal{HK}_g^1; \F_2)$, respectively, and sets us up to explicitly compute cup products of pairs of elements in these subspaces. Our approach detects new cohomology classes in $H^2(\mathcal{HI}_g^1)$ that cannot be detected rationally.

\subsection{The image of the BCJ homomorphism restricted to $\mathcal{HI}_g^1$ and $\mathcal{HK}_g^1$} When it is clear from context, we will abuse notation and denote the restriction maps $\sigma|_{\mathcal{HI}_g^1}$ and $\sigma|_{\mathcal{HK}_g^1}$ simply by $\sigma$. In Section \ref{section:handlebodyBCJ}, we define subspaces $B_2^{b_i} \subset B_2$ and $B_3^{b_i} \subset B_3$ by constructing a basis for each. We have the following theorem.

\begin{maintheorem} \label{theoremA}
    For $g \geq 3$, we have: \begin{itemize}
        \item $\sigma(\mathcal{HK}_g^1) = B_2^{b_i}$; and 
        \item $\sigma(\mathcal{HI}_g^1) = B_3^{b_i}$.
    \end{itemize}
\end{maintheorem}

By Proposition \ref{prop:imageBCJdimension}, the dimensions of $B_2^{b_i}$ and $B_3^{b_i}$ grow quadratically and cubically in $g$, matching the polynomial degrees of $\operatorname{dim}(B_2)$ and $\operatorname{dim}(B_3)$, respectively. At first glance, one might expect the images of these maps to be much smaller. As we outline in more detail in Section \ref{section:BCJ}, the BCJ homomorphism connects $\mathcal{I}_g^1$ to the Rochlin invariant. Specifically, let $h: \Sigma_g^1 \hookrightarrow \mathbb{S}^3$ be a Heegaard embedding and $k \in \mathcal{I}_g^1$. Cutting $\mathbb{S}^3$ open along $h(\Sigma_g^1)$ yields two handlebodies $\mathcal{A}$ and $\mathcal{B}$, and we obtain a homology $3$-sphere $M(h,k)$ by regluing these handlebodies according to $k$. The BCJ homomorphism evaluated on $k$ tells us how the Rochlin invariant of $M(h,k)$ varies as $h$ varies. If $k \in \mathcal{HI}_g^1$ and $\mathcal{V}_g^1$ is identified with either $\mathcal{A}$ or $\mathcal{B}$, then $M(h,k)$ is simply the $3$-sphere $\mathbb{S}^3$. Theorem \ref{theoremA} implies that there exist many Heegaard embeddings $h$ for which $M(h,k) \ncong \mathbb{S}^3$ and $M(h,k)$ has a nontrivial Rochlin invariant. \\

Inspired by the analogous result for $\mathcal{I}_g^1$, we conjecture the following.

\begin{conjecture} \label{conj:abelianize} For $g \gg 0$, we have 
    $H_1(\mathcal{HI}_g^1; \F_2) \cong B_3^{b_i}$.
\end{conjecture} 

\subsection{Cup products and second cohomology} By Theorem \ref{theoremA}, the BCJ homomorphism gives us surjections from $\mathcal{HI}_g^1$ and $\mathcal{HK}_g^1$ onto $B_3^{b_i}$ and $B_2^{b_i}$, respectively, which factor through $H_1(-; \F_2)$. This identifies the dual spaces $H^1(B_3^{b_i}; \F_2)$ and $H^1(B_2^{b_i}; \F_2)$ with subspaces of $H^1(\mathcal{HI}_g^1; \F_2)$ and $H^1(\mathcal{HK}_g^1; \F_2)$, respectively. The induced maps $\sigma^*$ on $H^2$ are exactly the restrictions of the cup product maps to pairs of elements in these subspaces. In our next main result, we explicitly calculate these cup products. 

\begin{maintheorem} \label{maintheorem:HI}
    For $g \geq 4$, the image of $$\sigma^*: H^2(B_3^{b_i}; \F_2) \rightarrow H^2(\mathcal{HI}_g^1; \F_2)$$ has dimension at least on the order of $g^6$. 
\end{maintheorem}

Theorem \ref{maintheorem:HI} is inspired by Brendle--Farb \cite[Proposition 9]{BrendleFarb}, who showed that the dimension of the image of $\sigma^*: H^2(B_3; \F_2) \rightarrow H^2(\mathcal{I}_g^1; \F_2)$ is also at least on the order of $g^6$. \\

In the following theorem, we explicitly compute cohomology classes which cannot be detected rationally, and so were not detected in \cite[Theorem A]{Holden}. 

\begin{maintheorem} \label{maintheorem:HK}
    Let $G$ be either $\mathcal{HI}_g^1$ or $\mathcal{HK}_g^1$. For $g \geq 4$, the induced map
        $$\sigma^*: H^2(B_2^{b_i}; \F_2) \rightarrow H^2(G; \F_2)$$ 
    has dimension at least on the order of $g^4$. 
\end{maintheorem}

Compare to \cite[Theorem 1]{BrendleFarb}, which says that the images of the analogous maps on $H^2(\mathcal{I}_g^1; \F_2)$ and $H^2(\mathcal{K}_g^1; \F_2)$ also have dimensions at least on the order of $g^4$. \\

Morita \cite{MoritaCasson} defined a surjective homomorphism $$\rho: \mathcal{K}_g^1 \rightarrow \mathcal{A}$$ that connects the Johnson kernel with the Casson invariant, a $\Z$-valued invariant on homology 3-spheres which reduces to the Rochlin invariant mod 2. Here, $\mathcal{A}$ is a $\Z$-algebra and $\rho$ is a lift of the BCJ homomorphism. The Morita homomorphism restricts to $\mathcal{HK}_g^1$ and we denote its image by $\mathcal{A}'$. Following the lifting techniques of Brendle--Farb \cite[Corollary 2 and Section 4]{BrendleFarb}, we have the following corollary.

\begin{corollary}
    For $g \geq 4$, the image of $$\rho^*: H^2(\mathcal{A}'; \Z) \rightarrow H^2(\mathcal{HK}_g^1; \Z)$$ has dimension at least on the order of $g^4$.
\end{corollary}

\subsection{Acknowledgments} I would like to thank my advisor Andrew Putman for introducing me to this problem and for his continued help and guidance. I thank him and Katherine Novey for their comments on an earlier draft of this paper. 

\section{Polynomial functions on quadratic forms}

In this section, we introduce the background needed to construct the BCJ homomorphism in Section \ref{section:BCJ}.

\subsection{Surface homology} \label{section:surfacehomology} Let $H := H_1(\Sigma_g^1; \F_2)$ and let $\hat{i}(-,-)$ be the $\F_2$-valued algebraic intersection form on $H$. 

\begin{figure}[H]
    \centering
    \includegraphics[scale=.5]{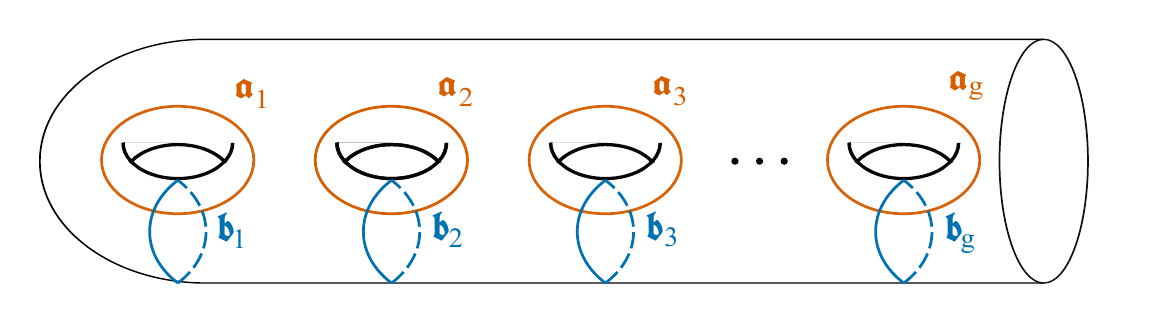}
    \caption{A basis for $H$}
    \label{figure:homologybasis}
\end{figure}

\noindent Figure \ref{figure:homologybasis} shows curves whose homology classes $[\mathfrak{a}_i] = a_i$ and $[\mathfrak{b}_i] = b_i$ form a symplectic basis for $H \cong \F_2 \langle a_1, \dots, a_g, b_1, \dots, b_g \rangle$. In Section \ref{section:handlebody}, when we begin studying the handlebody group, the $\mathfrak{b}_i$ curves will bound disks in our fixed handlebody $\mathcal{V}_g^1$.

\subsection{Quadratic forms} \label{section:quadraticforms} A \emph{quadratic form} is a set map $$\omega: H \rightarrow \F_2$$ such that, for each $x,y \in H$, we have $$\omega(x + y) = \omega(x) + \omega(y) + \hat{i}(x, y).$$

\noindent We define $\Omega$ to be the set of all quadratic forms on $H$. The symplectic group $\operatorname{Sp}_{2g}(\F_2)$ acts naturally on $\Omega$ through its action on $H$. Given $M \in \operatorname{Sp}_{2g}(\F_2)$ and $\omega \in \Omega$ and $x \in H$, we have $$(M \cdot \omega)(x) = \omega(M(x)).$$ 

\subsection{Polynomial functions on $\Omega$} \label{section:polyfunctions} For each $x \in H$, we get a linear function $$\bar{x}: \Omega \rightarrow \F_2$$ given by $$\bar{x}(\omega) = \omega(x).$$ Our set map $x \mapsto \bar{x}$ fails to be a homomorphism in the following way. Let $1$ denote the constant function that sends all elements of $\Omega$ to $1 \in \F_2$. Since \begin{gather*}
    \overline{(x+y)}(\omega) = \omega(x+y) \\
    = \omega(x) + \omega(y) + \hat{i}(x,y) \\ 
    = (\bar{x} + \bar{y} + \hat{i}(x,y))(\omega),
\end{gather*} we have the following lemma.

\begin{lemma}[{\cite[Lemma 4]{BCJ}}]
    Let $x,y \in H$. We have $$\overline{(x+y)} = \bar{x} + \bar{y} + \hat{i}(x, y).$$
\end{lemma} The set $\{ 1, \bar{a}_1, \dots, \bar{a}_g, \bar{b}_1, \dots, \bar{b}_g \}$ forms a basis for the $\F_2$-vector space of linear functions\footnote{See \cite[Section 4]{BCJ}, which proves this from the perspective of $\Omega$ being an affine space over $H^1(\Sigma_g^1; \F_2)$.} on $\Omega$. A polynomial function on $\Omega$ is a sum of products of linear functions. Since any such polynomial and its square evaluate in the same way, we consider the space $B_r$ of \emph{Boolean} (``square-free'') polynomials of degree at most $r$.

\subsubsection{The action of $\operatorname{Sp}_{2g}(\F_2)$} \label{section:llspaction} The action of $\operatorname{Sp}_{2g}(\F_2)$ on $B_r$ is adjoint to its action on $\Omega$ described in Section \ref{section:quadraticforms}. In particular, for $x \in H$ and $M \in \operatorname{Sp}_{2g}(\F_2)$, we have $$M \cdot \bar{x} = \overline{M(x)}.$$

\section{The Birman--Craggs--Johnson homomorphism} \label{section:BCJ}

\subsection{Heegaard embeddings and quadratic forms} \label{section:HeegaardandQFs} Let $h: \Sigma_g^1 \hookrightarrow \mathbb{S}^3$ be a Heegaard embedding. Associated to $h$ is a quadratic form $\omega_h: H \rightarrow \F_2$ coming from Seifert's linking form. Briefly, given $x \in H$, the form $\omega_h$ measures the self-linking of $h(x)$ and its positive push-off $h(x)^+$. See \cite[Section 5]{BCJ} for more details. All quadratic forms arise in this way: given $\omega \in \Omega$, we have $\omega = \omega_h$ for some Heegaard embedding $h$. 

\subsection{The Birman--Craggs homomorphisms} \label{section:BChomoms} 

Given a Heegaard embedding $h$ and $k \in \mathcal{I}_g^1$, we form a 3-manifold $M(h,k)$ as follows. Cut $\mathbb{S}^3$ open along $h(\Sigma_g^1)$. This splits $\mathbb{S}^3$ into two handlebodies $\mathcal{A}$ and $\mathcal{B}$ with $\partial \mathcal{A} = \partial \mathcal{B} = h(\Sigma_g^1)$. We reglue $\mathcal{A}$ to $\mathcal{B}$ via $k$ by sending $x \in \partial \mathcal{A}$ to $hkh^{-1}(x) \in \partial \mathcal{B}$. The resulting manifold $M(h,k)$ is a homology 3-sphere, and so we can take its Rochlin invariant $\mu(M(h,k)) \in \F_2$. Given a fixed Heegaard embedding $h$, Birman--Craggs \cite{BC} proved the remarkable fact that the map $\mathcal{I}_g^1 \rightarrow \F_2$ given by $k \mapsto \mu(M(h,k))$ is a homomorphism. Johnson \cite{BCJ} proved that these homomorphisms depend entirely on the induced quadratic form $\omega_h$ (recall Section \ref{section:HeegaardandQFs}) and so are in one-to-one correspondence with $\Omega$, the set of all quadratic forms on $H$. From this, Johnson enumerated the Birman--Craggs homomorphisms and then packaged them into one homomorphism which we describe in the following section.

\subsection{The Birman--Craggs--Johnson homomorphism} Let $B_3$ be the vector space of degree-3 Boolean polynomials as defined in Section \ref{section:polyfunctions}. Johnson \cite{BCJ} defined a surjective homomorphism $$\sigma: \mathcal{I}_g^1 \rightarrow B_3,$$ now called the Birman--Craggs--Johnson (BCJ) homomorphism. The BCJ homomorphism packages all of the Birman--Craggs homomorphisms together in the sense that $\operatorname{ker}(\sigma)$ is exactly the intersection of the kernels of all of the Birman--Craggs homomorphisms.

To see this more clearly, consider the following construction. Enumerate the Birman--Craggs homomorphisms $\rho_1, \dots, \rho_N$, and consider the homomorphism $$\tilde{\sigma}: \mathcal{I}_g^1 \rightarrow \bigoplus_{i=1}^N \F_2$$ defined by $$k \mapsto (\rho_1(k), \dots, \rho_N(k)).$$ The image $\operatorname{im}(\tilde{\sigma})$ is an $\F_2$-vector space with kernel $\operatorname{ker}(\tilde{\sigma}) = \bigcap_{i=1}^N \operatorname{ker}(\rho_i)$, and Johnson showed that $\operatorname{im}(\tilde{\sigma}) \cong B_3$.

\section{Handlebody groups} \label{section:handlebody}

Fix a handlebody $\mathcal{V}_g^1$ such that $\partial \mathcal{V}_g^1 = \Sigma_g^1$ and the curves $\mathfrak{b}_1, \dots, \mathfrak{b}_g$ from Figure \ref{figure:homologybasis} bound disks. Recall from Section \ref{section:introhandlebody} that the handlebody group $\mathcal{H}_g^1$ is the subgroup of $\Mod_g^1$ that extends to $\mathcal{V}_g^1$. 

\subsection{Symplectic representation} The handlebody group acts on $H_1(\Sigma_g^1; \Z)$, and Hirose \cite{Hirose} proved that its image in $\operatorname{Sp}_{2g}(\Z)$ is $$\ell \ell Sp_{2g}(\Z) = \left\{  \begin{bmatrix} A & 0 \\ B & (A^t)^{-1} \end{bmatrix} \; | \; A \in \operatorname{GL}_g(\Z), \; B = B^t \right\}.$$

\noindent Hirose chose the convention that the $\mathfrak{a}_i$ curves bound disks in the fixed handlebody, so he called this image $urSp_{2g}(\Z)$, where \emph{ur} stands for ``upper right." We instead have the $\mathfrak{b}_i$ curves bound disks in $\mathcal{V}_g^1$, so the nontrivial entries of our matrices lie in the lower left.\footnote{We choose this convention so that $\mathcal{H}_g^1$ acts in the standard way on $H_1(\mathcal{V}_g^1; \Z) \cong \Z \langle a_1, \dots, a_g \rangle$.}

\subsection{The handlebody Torelli group and handlebody Johnson kernel} As discussed in Section \ref{section:intro}, the handlebody Torelli group $\mathcal{HI}_g^1 := \mathcal{H}_g^1 \cap \mathcal{I}_g^1$ is the subgroup of $\mathcal{H}_g^1$ that acts trivially on $H_1(\Sigma_g^1; \Z)$. This is summarized in the short exact sequence $$1 \rightarrow \mathcal{HI}_g^1 \rightarrow \mathcal{H}_g^1 \rightarrow \ell \ell Sp_{2g}(\Z) \rightarrow 1.$$ We are also interested in the handlebody Johnson kernel $\mathcal{HK}_g^1 := \mathcal{H}_g^1 \cap \mathcal{K}_g^1$.

\section{The image of the BCJ homomorphism} \label{section:handlebodyBCJ}

In this section, we calculate the images of the BCJ homomorphism $\sigma$ restricted to $\mathcal{HK}_g^1$ and $\mathcal{HI}_g^1$. Building on the setup in Section \ref{section:polyfunctions}, define $B_2^{b_i}$ to be the subspace of $B_2$ generated by the following elements:
\begin{itemize}
\item $\ai \bi$ for $1 \leq i \leq g$; and
\item $\ai \bj$ for $1 \leq i \neq j \leq g$; and
\item $\bi$ for $1 \leq i \leq g$; and
\item $\bi \bj$ for $1 \leq i < j \leq g$.
\end{itemize} For future use, we denote by $\mathcal{B}_2^{b_i}$ the set of these basis elements. Similarly, define $B_3^{b_i}$ to be the subspace of $B_3$ generated by $\mathcal{B}_2^{b_i}$ along with the following degree-3 elements:
\begin{itemize}
    \item $\ai \aj \bi$ for $1 \leq i \neq j \leq g$; and
    \item $\ai \aj \bk$ for $1 \leq i < j \leq g$ and $k \neq i,j$; and
    \item $\ai \bi \bj$ for $1 \leq i \neq j \leq g$; and
    \item $\ai \bj \bk$ for $1 \leq i \leq g$ and $i \neq j,k$ and $1 \leq j < k \leq g$; and
    \item $\bi \bj \bk$ for $1 \leq i < j < k \leq g$.
\end{itemize} We denote the set of basis elements for $B_3^{b_i}$ by $\mathcal{B}_3^{b_i}$. In particular, $B_3^{b_i}$ does not contain elements of the forms $1$ and $\ai$ and $\ai \aj$ and $\ai \aj \ak$ of $B_3$ with ``no $b_i$'s." We are now ready to state Theorem A.

\begin{theorem*}[Theorem \ref{theoremA}] For $g \geq 3$, we have: \begin{itemize}
        \item $\sigma(\mathcal{HK}_g^1) = B_2^{b_i}$; and 
        \item $\sigma(\mathcal{HI}_g^1) = B_3^{b_i}$. \end{itemize}  
\end{theorem*}

Theorem \ref{theoremA} will follow from Lemma \ref{lemma:subset} and Lemma \ref{lemma:llspsubspace}. We first establish containment with the following lemma.

\begin{lemma} \label{lemma:subset}
    For $g \geq 3$, we have: \begin{itemize}
        \item $B_2^{b_i} \subset \sigma(\mathcal{HK}_g^1)$; and 
        \item $B_3^{b_i} \subset \sigma(\mathcal{HI}_g^1)$. \end{itemize}  
\end{lemma}

As we discuss in more detail in Section \ref{section:llsp}, the groups $\sigma(\mathcal{HK}_g^1)$ and $\sigma(\mathcal{HI}_g^1)$ are closed under a natural action of $\llsp$. There is a mapping class $T_y T_{y'}^{-1}$ that normally generates $\mathcal{HI}_g^1$ in $\mathcal{H}_g^1$, and $\sigma(\mathcal{HI}_g^1)$ is exactly the $\llsp$-subspace of $B_3$ generated by $\sigma(T_y T_{y'}^{-1})$. We have the following lemma.

\begin{lemma} \label{lemma:llspsubspace}
    For $g \geq 3$, the space $B_3^{b_i}$ is an $\llsp$-subspace of $B_3$ which contains $\sigma(T_y T_{y'}^{-1})$.
\end{lemma} 

Before we prove these lemmas, we describe in Section \ref{section:explicitformulas} how to explicitly evaluate the BCJ homomorphism on certain types of mapping classes. Then, in Section \ref{section:llsp}, we describe the action of $\llsp$ on $B_2^{b_i}$ and $B_3^{b_i}$.

\subsection{Explicit formulas} \label{section:explicitformulas}

A \emph{separating disk twist} is a Dehn twist $T_x$ about a separating simple closed curve $x$ on $\Sigma_g^1$ that separates $\Sigma_g^1$ and bounds an embedded disk in $\mathcal{V}_g^1$. Since all disk twists lie in $\mathcal{H}_g^1$ and all separating twists lie in $\mathcal{K}_g^1$, all separating disk twists lie in $\mathcal{HK}_g^1$. We have the following lemma.

\begin{lemma}[{Johnson, \cite[Lemma 12a]{BCJ}}] \label{lemma:septwist} Let $T_x$ be a separating disk twist and let $S_x$ be the component of $\Sigma_g^1 - x$ that does not contain $\partial \Sigma_g^1$. Choose a maximal symplectic basis $c_1, d_1, \dots, c_k, d_k$ for $H_1(S_x; \F_2) \subset H$. Then the BCJ homomorphism evaluates as $$\sigma(T_x) = \sum_{i=1}^k \bar{c}_i \bar{d}_i,$$ which does not depend on our choice of symplectic basis.
    
\end{lemma}

A \emph{bounding pair annulus twist} is a product $T_y T_{y'}^{-1}$ for which $y$ and $y'$ are disjoint nonseparating simple closed curves on $\Sigma_g^1$ such that $y \cup y'$ separates $\Sigma_g^1$ and bounds an embedded annulus in $\mathcal{V}_g^1$. All bounding pair annulus twists lie in $\mathcal{HI}_g^1$ (in fact, they generate the entire group -- see \cite{Omori}). By \cite[Theorem 1.2]{Omori}, the bounding pair annulus twist $T_y T_{y'}^{-1}$ in Figure \ref{figure:curvesforBCJimage} normally generates $\mathcal{HI}_g^1$ in $\mathcal{H}_g^1$. We have the following lemma.

\begin{lemma}[{Johnson, \cite[Lemma 12b]{BCJ}}] Let $T_{y} T_{y'}^{-1}$ be a bounding pair annulus twist and let $S_{y,y'}$ be the component of $\Sigma_g^1 - (y \cup y')$ that does not contain $\partial \Sigma_g^1$. Choose a maximal symplectic basis $c_1, d_1, \dots, c_k, d_k$ for $H_1(S_{y,y'}; \F_2) \subset H$. Then the BCJ homomorphism evaluates as $$\sigma(T_y T_{y'}^{-1}) = \Big( \sum_{i=1}^k \bar{c}_i \bar{d}_i \Big)(\overline{[y]} + 1),$$ which does not depend on our choice of symplectic basis.
    
\end{lemma}

\subsection{Action of $\llsp$} \label{section:llsp} The map $\sigma: \mathcal{HI}_g^1 \rightarrow B_3$ is $\mathcal{H}_g^1$-equivariant in the following way. The handlebody group acts on its normal subgroup $\mathcal{HI}_g^1$ by conjugation, and it acts on $B_3$ via its action on $H := H_1(\Sigma_g^1; \F_2)$. 
The analogous statement holds for $\mathcal{HK}_g^1$. In particular, the subspaces $\sigma(\mathcal{HI}_g^1)$ and $\sigma(\mathcal{HK}_g^1)$ are closed under the action of $\llsp$ as described in Section \ref{section:llspaction}. We have the following proposition.

\begin{proposition}[Generators for $\llsp$] \label{prop:llSpgenerators}
    The group $\llsp$ is generated by the symplectic automorphisms $$X_{ij}: \; a_j \mapsto a_j + a_i, \; \; b_i \mapsto b_i + b_j$$ for $1 \leq i \neq j \leq g$; and $$Y_{ij}: \; a_i \mapsto a_i + b_j, \; \; a_j \mapsto a_j + b_i$$ for\footnote{Note that we include automorphisms of the form $Y_{ii}$.} $1 \leq i,j \leq g$; where basis elements of $H$ not mentioned are assumed to be fixed.
\end{proposition}

See \cite[Chapter 2.2]{Omeara}.

\subsection{Proof of Lemma \ref{lemma:subset}} We are now ready to prove Lemma \ref{lemma:subset}.

\begin{proof} First, we show $B_2^{b_i} \subset \sigma(\mathcal{HK}_g^1)$. 

\begin{figure}[htbp]
    \centering
    \includegraphics[scale=.5]{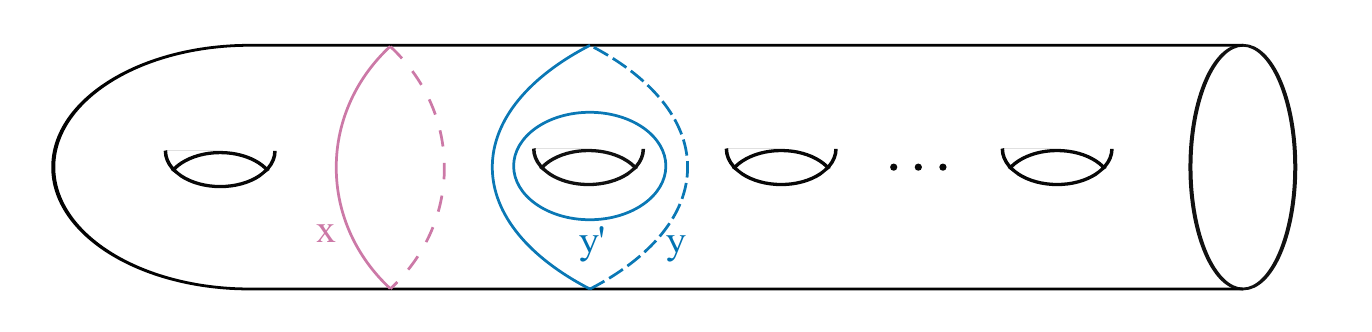}
    \caption{A separating disk twist $T_x$ and a bounding pair annulus twist $T_y T_{y'}^{-1}$}
    \label{figure:curvesforBCJimage}
\end{figure}

Figure \ref{figure:curvesforBCJimage} shows a separating simple closed curve $x$ that bounds a disk in $\mathcal{V}_g^1$, so by Lemma \ref{lemma:septwist} we have $$\sigma(T_x) = \bar{a}_1 \bar{b}_1.$$ 
We leverage the fact that $\operatorname{im}(\sigma)$ is closed under the action of $\llsp$. Since $\llsp$ contains all automorphisms that permute generators\footnote{Such an automorphism can be realized by an element of $\mathcal{H}_g^1$ that interchanges the $i^{\text{th}}$ and $j^{\text{th}}$ handles of the handlebody. See \cite{Suzuki}.} $a_i \leftrightarrow a_j$ and $b_i \leftrightarrow b_j$, the fact that $\bar{a}_1 \bar{b}_1 \in \sigma(\mathcal{HK}_g^1)$ implies $\boxed{\ai \bi} \in \sigma(\mathcal{HK}_g^1)$ for all $1 \leq i \leq g$. Next, we calculate $$X_{ij}(\ai \bi) = \ai \bi + \ai \bj.$$ Since $\ai \bi \in \sigma(\mathcal{HK}_g^1)$, we must have $\boxed{\ai \bj} \in \sigma(\mathcal{HK}_g^1)$ as well. We now show that the remaining basis elements of $B_2^{b_i}$ lie in $\sigma(\mathcal{HK}_g^1)$, building on previous work as we go. We have $$Y_{ik}(\ak \bj) = \ak \bj + \boxed{\bi \bj}$$ and $$Y_{jj}(\aj \bi) = \aj \bi + \bi \bj + \boxed{\bi}.$$ We have now shown $B_2^{b_i} \subset \sigma(\mathcal{HK}_g^1)$. It remains to show that the degree-3 generators of $B_3^{b_i}$ lie in $\sigma(\mathcal{HI}_g^1)$. Consider the bounding pair annulus twist $T_y T_{y'}^{-1}$ in Figure \ref{figure:curvesforBCJimage}. We have $$\sigma(T_y T_{y'}^{-1}) = \bar{a}_1 \bar{b}_1 \bar{a}_2 + \bar{a}_1 \bar{b}_1,$$ so all basis elements of the form $\boxed{\ai \aj \bi}$ lie in $\sigma(\mathcal{HI}_g^1)$. Next, we have \begin{gather*}
X_{ik}(\ai \aj \bi) = \ai \aj \bi + \boxed{\ai \aj \bk} \\
Y_{jj}(\ai \aj \bi) = \ai \aj \bi + \boxed{\ai \bi \bj} + \ai \bi \\
Y_{jj}(\ai \aj \bk) = \ai \aj \bk + \boxed{\ai \bj \bk} + \ai \bk \\
Y_{ik}(\ai \bi \bj) = \ai \bi \bj + \boxed{\bi \bj \bk}. \qedhere \end{gather*} \end{proof}

\subsection{Proof of Lemma \ref{lemma:llspsubspace}} We are now ready to prove Lemma \ref{lemma:llspsubspace}.

\begin{proof} Recall our basis $\mathcal{B}_3^{b_i}$ for $B_3^{b_i}$ defined at the beginning of Section \ref{section:handlebodyBCJ}. We also recall from Proposition \ref{prop:llSpgenerators} that $\llsp$ is generated by the automorphisms $X_{ij}$ (for $1 \leq i \neq j \leq g$) and $Y_{ij}$ (for $1 \leq i, j \leq g$). It suffices to show that, for each $v \in \mathcal{B}_3^{b_i}$, we can express all $X_{ij}(v)$ and $Y_{ij}(v)$ as linear combinations of elements in $\mathcal{B}_3^{b_i}$. 

$\boxed{\ai \bi}$: We have 
\begin{gather*}
    X_{ij}(\ak \bk) = \begin{cases} 
     \ak \bk & i,j \neq k \\
     \ak \overline{(b_k + b_j)} = \ak \bk + \ak \bj & i=k \\
     \overline{(a_k + a_i)} \bk = \ak \bk + \ai \bk & j=k
   \end{cases}  \end{gather*} and \begin{gather*}
    Y_{ij}(\ak \bk) = \begin{cases} 
     \ak \bk & i,j \neq k \\
     \overline{(a_k + b_j)} \bk = \ak \bk + \bj \bk & i=k \text{ and } j \neq k \\
     \overline{(a_k + b_i)} \bk = \ak \bk + \bi \bk & i \neq k \text{ and } j=k \\ 
     \overline{(a_k + b_k)} \bk = (\ak + \bk +1) \bk = \ak \bk & i=j=k \end{cases}.  \end{gather*}

$\boxed{\ai \bj}$: We have \begin{gather*}
    X_{ij}(\ak \bl) = \begin{cases} 
     \ak \bl & i \neq \ell \text{ and } j \neq k  \\
     \ak \overline{(b_{\ell} + b_j)} = \ak \bl + \ak \bj & i= \ell \text{ and } j \neq k \\
     \overline{(a_k + a_i)} \bl = \ak \bl + \ai \bl & i \neq \ell \text{ and } j=k \\
     \overline{(a_k + a_{\ell})} \; \overline{(b_{\ell} + b_k)} = \ak \bl + \ak \bk + \al \bl + \al \bk & i=\ell \text{ and } j=k
   \end{cases}  \end{gather*} and \begin{gather*}
    Y_{ij}(\ak \bl) = \begin{cases} 
     \ak \bl & i,j \neq k  \\
     \overline{(a_k + b_j)} \bl = \ak \bl + \bj \bl & i=k \text{ and } j \neq k \\
     \overline{(a_k + b_i)} \bl = \ak \bl + \bi \bl & i \neq k \text{ and } j=k \\
     \overline{(a_k + b_k)} \bl = (\ak + \bk + 1) \bl = \ak \bl + \bk \bl + \bl & i=j=k
   \end{cases}.  \end{gather*}

$\boxed{\bi}$: We have \begin{gather*}
    X_{ij}(\bk) = \begin{cases} 
     \bk & i \neq k \\
     \bk + \bj & i=k
   \end{cases} 
\end{gather*} and $Y_{ij}(\bk) = \bk.$

$\boxed{\bi \bj}$: We have \begin{gather*}
    X_{ij}(\bk \bl) = \begin{cases} 
     \bk \bl & i \neq k, \ell  \\
     \bk \bl + \bj \bl & i=k \\
     \bk \bl + \bj \bk & i=\ell
   \end{cases}  \end{gather*} and $Y_{ij}(\bk \bl) = \bk \bl.$

$\boxed{\bi \bj \bk}$: We have \begin{gather*}
    X_{ij}(\bk \bl \bm) = \begin{cases} 
     \bk \bl \bm & i \neq k, \ell, m  \\
     \bk \bl \bm + \bj \bl \bm  & i=k \\
     \bk \bl \bm + \bj \bk \bm & i=\ell \\
     \bk \bl \bm + \bj \bk \bl & i=m
   \end{cases}  \end{gather*} and $Y_{ij}(\bk \bl \bm) = \bk \bl \bm.$ 

$\boxed{\ai \aj \bi}$: We have \begin{gather*}
    X_{ij}(\ak \al \bk) = \begin{cases} 
     \ak \al \bk & i \neq k \text{ and } j \neq k, \ell  \\
     \ak \al \bk + \ai \al \bk & j=k \\
     \ak \al \bk + \ai \ak \bk & i \neq k \text{ and } j=\ell \\
    \ak \al \bk + \ak \al \bj & i=k \text{ and } j \neq \ell \\
     \ak \al \bk + \ak \al \bl + \ak \bk + \ak \bl  & i=k \text{ and } j=\ell \\
   \end{cases}.  \end{gather*} Now, observe that the map $Y_{ij}$ acts on $\ak$ in the following way:  \begin{equation} \label{eq:Yij(ak)} Y_{ij}(\ak) = \begin{cases} 
   \ak & i,j \neq k \\
   \ak + \bj & i=k \text{ and } i \neq j \\
   \ak + \bi & i \neq j \text{ and } j=k \\
   \ak + \bk +1 & i=j=k
   \end{cases}. \end{equation} The map $Y_{ij}$ acts on $\al$ in the analogous way. So then we have $$Y_{ij}(\ak \al \bk) = \bar{c} \; \bar{d} \; \bk$$ for $\bar{c} \in \{\ak, \; \ak + \bj, \; \ak + \bi, \; \ak + \bk +1 \}$ and $\bar{d} \in \{ \al, \; \al + \bj, \; \al + \bi, \; \al + \bl +1 \}$, and so $Y_{ij}(\ak \al \bk)$ is a linear combination of elements in $\mathcal{B}_3^{b_i}$. 

$\boxed{\ai \aj \bk}$: We have \begin{gather*}
    X_{ij}(\ak \al \bm) = \begin{cases} 
     \ak \al \bm & i \neq m \text{ and } j \neq k, \ell  \\
     \ak \al \bm + \ai \al \bm & i \neq m \text{ and } j=k \\
     \ak \al \bm + \ai \ak \bm & i \neq m \text{ and } j=\ell \\
     \ak \al \bm + \ak \al \bj & i=m \text{ and } j \neq k,\ell \\
     \ak \al \bm + \ak \al \bk + \al \am \bm + \al \am \bk  & i=m \text{ and } j=k \\
     \ak \al \bm + \ak \al \bl + \ak \am \bm + \ak \am \bl  & i=m \text{ and } j=\ell \\
   \end{cases}.  \end{gather*} By Equation \eqref{eq:Yij(ak)}, we have $$Y_{ij}(\ak \al \bm) = \bar{c} \; \bar{d} \; \bm$$ for $\bar{c} \in \{\ak, \; \ak + \bj, \; \ak + \bi, \; \ak + \bk +1 \}$ and $\bar{d} \in \{ \al, \; \al + \bj, \; \al + \bi, \; \al + \bl +1 \}$, and so $Y_{ij}(\ak \al \bm)$ is a linear combination of elements in $\mathcal{B}_3^{b_i}$. 

$\boxed{\ai \bi \bj}$: We have \begin{gather*}
    X_{ij}(\ak \bk \bl) = \begin{cases} 
     \ak \bk \bl & i,j \neq k \text{ and } i \neq \ell \\
     \ak \bk \bl + \ai \bk \bl & i \neq \ell \text{ and } j=k \\
     \ak \bk \bl + \ak \bj \bl & i=k \\
     \ak \bk \bl + \ak \bj \bk & i=\ell \text{ and } j \neq k \\
     \ak \bk \bl + \ak \bk + \al \bk \bl + \al \bk & i=\ell \text{ and } j=k
   \end{cases}  \end{gather*} and \begin{gather*}
       Y_{ij}(\ak \bk \bl) = \begin{cases} 
     \ak \bk \bl & i,j \neq k \\
     \ak \bk \bl + \bj \bk \bl & i=k \text{ and } j \neq k \\
     \ak \bk \bl + \bi \bk \bl & i \neq k \text{ and } j = k \\
     \ak \bk \bl & i=j=k
   \end{cases}.
   \end{gather*} 
   
$\boxed{\ai \bj \bk}$: We have \begin{gather*}
    X_{ij}(\ak \bl \bm) = \begin{cases} 
     \ak \bl \bm & i \neq \ell, m \text{ and } j \neq k \\
     \ak \bl \bm + \ak \bj \bm & i=\ell \text{ and } j \neq k \\
     \ak \bl \bm + \ak \bl \bj & i=m \text{ and } j \neq k \\
     \ak \bl \bm + \ai \bl \bm & i \neq \ell,m \text{ and } j=k \\
     \ak \bl \bm + \ak \bk \bm + \al \bl \bm + \al \bk \bm & i=\ell \text{ and } j=k \\
     \ak \bl \bm + \ak \bk \bl + \am \bl \bm + \am \bk \bl & i=m \text{ and } j=k
   \end{cases}
\end{gather*} and \begin{gather*}
    Y_{ij}(\ak \bl \bm) = \begin{cases} 
     \ak \bl \bm & i,j \neq k \\
     \ak \bl \bm + \bj \bl \bm & i=k \text{ and } j \neq k \\
     \ak \bl \bm + \bi \bl \bm & i \neq k \text{ and } j=k \\
     \ak \bl \bm + \bk \bl \bm + \bl \bm & i=j=k
   \end{cases} . \qedhere
\end{gather*} \end{proof}

\subsection{Dimension counting} We have the following proposition.

\begin{proposition} \label{prop:imageBCJdimension}
    For $g \geq 3$, we have $$\operatorname{dim}(\sigma(\mathcal{HI}_g^1)) = \frac{7g^3 + 5g}{6} \quad \quad \text{and} \quad \quad \operatorname{dim}(\sigma(\mathcal{HK}_g^1)) = \frac{3g^2 + g}{2}.$$ 
\end{proposition}

\begin{proof} By Theorem \ref{theoremA}, we have $\sigma(\mathcal{HK}_g^1) = B_2^{b_i}$ and $\sigma(\mathcal{HI}_g^1) = B_3^{b_i}$. For $r \in \{2,3 \}$, we compare the dimension of $B_r^{b_i}$ with that of $B_r$. We have $$\operatorname{dim}(B_r) = \sum_{i=0}^r \binom{2g}{i}.$$ See \cite[Theorem 6]{BCJ}. As defined at the beginning of Section \ref{section:handlebodyBCJ}, the spaces $B_2^{b_i}$ and $B_3^{b_i}$ have basis sets $\mathcal{B}_2^{b_i}$ and $\mathcal{B}_3^{b_i}$, respectively. A basis for $B_2$ is $$\mathcal{B}_2^{b_i} \cup \Big\{1, \; \ai \; (1 \leq i \leq g), \; \ai \aj \; (1 \leq i < j \leq g) \Big\} $$ and a basis for $B_3$ is $$\mathcal{B}_3^{b_i} \cup \Big\{1, \; \ai \; (1 \leq i \leq g), \; \ai \aj \; (1 \leq i < j \leq g), \; \ai \aj \ak \; (1 \leq i < j < k \leq g) \Big\}. $$ In particular, $B_r$ has $\sum_{i=0}^r {g \choose i}$ more generators than $B_r^{b_i}$. Hence $B_r^{b_i}$ has dimension $$\sum_{i=0}^r \left( \binom{2g}{i} - \binom{g}{i} \right),$$ and our result follows. \end{proof}

\section{Group cohomology and abelian cycles} \label{section:homology}

Let $G = \mathcal{HI}_g^1$ and $B = B_3^{b_i}$ or $G = \mathcal{HK}_g^1$ and $B = B_2^{b_i}$. By Theorem \ref{theoremA}, we have the surjective BCJ homomorphism $\sigma: G \rightarrow B$. For the remainder of the paper, we study the induced map $$\sigma^*: H^2(B; \F_2) \rightarrow H^2(G; \F_2)$$ following Brendle--Farb \cite{BrendleFarb}. To do this, we dualize and study the induced map on homology $$\sigma_*: H_2(G; \F_2) \rightarrow H_2(B; \F_2).$$ The group $H_2(B; \F_2)$ is isomorphic to the dual of $H^2(B; \F_2)$, and if $\xi$ lies in $\operatorname{im}(\sigma_*) \subset H_2(B; \F_2)$, then its dual $\xi^{\star}$ is \emph{not} in $\operatorname{ker}(\sigma^*) \subset H^2(B; \F_2)$. See \cite[Section 5]{BrendleFarb}. By constructing elements in $\operatorname{im}(\sigma_*)$, we bound from below the size of $\operatorname{im}(\sigma^*)$ as in Theorems \ref{maintheorem:HK} and \ref{maintheorem:HI}. As we set up in the following subsections, we use the method of \emph{abelian cycles}. Pairs of commuting elements $f,h \in G$ will give us abelian cycles $\{ f,h \} \in H_2(G; \F_2)$. For many abelian cycles, we compute $\sigma_*(\{ f,h \}) = \xi$ in $\operatorname{im}(\sigma_*)$. The dual $\xi^{\star}$ then maps under $\sigma^*$ to the corresponding cup product in $H^2(G; \F_2)$.

\subsection{Homology of $B$} Recall that we have $B = B_2^{b_i}$ or $B = B_3^{b_i}$. Since\footnote{We calculate $N$ in Proposition \ref{prop:imageBCJdimension}.} $B \cong \bigoplus_N \F_2$, we have $H_1(B; \Z) \cong B$ and $\operatorname{Tor}(B, \F_2) \cong B$, and the Universal Coefficient Theorem gives us $$H_2(B; \F_2) \cong (\bigwedge^2 B) \oplus B.$$

\subsection{Abelian cycles} If $f,h \in G$ commute, there is a map $i: \Z^2 \rightarrow G$ sending the standard generators of $\Z^2$ to $f$ and $h$. This induces a map on homology $$i_*: H_2(\Z^2; \F_2) \rightarrow H_2(G; \F_2).$$ The image of the fundamental class $1$ of $H_2(\Z^2; \F_2) \cong \F_2$ is the \emph{abelian cycle} $$\{ f,h \} := i_*(1).$$ The upshot is that commuting elements of $G$ give us (possibly trivial) elements of $H_2(G; \F_2)$. To find elements in $\operatorname{im}(\sigma_*)$, we apply $\sigma_*$ to abelian cycles. By \cite[Lemma 2.1]{Sakasai1}, this calculation is $$\sigma_*(\{f,h \}) = (\sigma(f) \wedge \sigma(h), 0) \in H_2(B; \F_2).$$ The class $\sigma_*( \{f,h \})$ always vanishes in the Tor term of $H_2(B; \F_2)$. See \cite[Section 3.1]{BrendleFarb}. Because of this, we denote classes $(\sigma(f) \wedge \sigma(h), 0) \in H_2(B; \F_2)$ simply as $\sigma(f) \wedge \sigma(h) \in \bigwedge^2 B$.

\section{Cohomology classes detected by $B_2^{b_i}$} \label{section:HK}

In this section, we prove Theorem \ref{maintheorem:HK}, which we now recall.

\begin{theorem*}[Theorem \ref{maintheorem:HK}]
Let $G$ be either $\mathcal{HI}_g^1$ or $\mathcal{HK}_g^1$. For $g \geq 4$, the induced map
        $$\sigma^*: H^2(B_2^{b_i}; \F_2) \rightarrow H^2(G; \F_2)$$ 
    has dimension at least on the order of $g^4$. 
\end{theorem*} 

Theorem \ref{maintheorem:HK} will follow from Theorem \ref{theorem:johnsonkernel} and Proposition \ref{prop:dimageHK}, which concern the image of the induced map on homology $$\sigma_*: H_2(\mathcal{HK}_g^1; \F_2) \rightarrow H_2(B_2^{b_i}; \F_2)$$ following the setup in Section \ref{section:homology}. In particular, we define a certain subspace of $\bigwedge^2 B_2^{b_i} \subset H_2(B_2^{b_i}; \F_2)$ and show it is contained in $\operatorname{im}(\sigma_*)$.

\subsection{Index-matched elements} \label{section:IM2} Recall from Section \ref{section:handlebodyBCJ} that a basis for $B_2^{b_i}$ is $$\mathcal{B}_2^{b_i} = \{ \ai \bi \; (1 \leq i \leq g), \; \ai \bj \; (1 \leq i \neq j \leq g), \; \bi \; (1 \leq i \leq g), \; \bi \bj \; (1 \leq i < j \leq g) \}.$$ Then a basis for $\bigwedge^2 B_2^{b_i}$ is made up of ${|\mathcal{B}_2^{b_i}| \choose 2}$ elements of the form $u \wedge v$, where $u,v \in \mathcal{B}_2^{b_i}$ are distinct. We say a basis element of $\bigwedge^2 B_2^{b_i}$ is \emph{index-matched} if it is of the form\footnote{We allow $\bar{y} = \bi$ to include elements of the form $\ai \bar{x} \wedge \bi$.} $\ai \bar{x} \wedge \bi \bar{y}$. We are now ready to state Theorem \ref{theorem:johnsonkernel}.

\begin{theorem} \label{theorem:johnsonkernel}
    For $g \geq 4$, the image of the map $\sigma_*: H_2(\mathcal{HK}_g^1; \F_2) \rightarrow H_2(B_2^{b_i}; \F_2)$ contains the subspace spanned by all basis elements of $\bigwedge^2 B_2^{b_i} \subset H_2(B_2^{b_i}; \F_2)$ which are not index-matched.
\end{theorem}

\begin{proof} Recall that $\llsp$ contains automorphisms that permute indices of generators of $H$. It then follows, for example, that $\bar{a}_1 \bar{b}_1 \wedge \bar{a}_2 \bar{b}_2 \in \operatorname{im}(\sigma_*)$ implies $\ai \bi \wedge \aj \bj \in \operatorname{im}(\sigma_*)$ for all choices of $1 \leq i,j \leq g$. To prove our theorem, it suffices to show that $\operatorname{im}(\sigma_*)$ contains basis elements of each of the following 16 types (where we assume indices are distinct):
\begin{multicols}{4}
 \begin{enumerate}
        \item $\ai \bi \wedge \aj \bj$
        \item $\ai \bi \wedge \aj \bk$
        \item $\ai \bi \wedge \bj \bk$
        \item $\ai \bj \wedge \ai \bk$
        \item $\ai \bj \wedge \ak \bj$
        \item $\ai \bj \wedge \ak \bl$
        \item $\ai \bj \wedge \bj \bk$
        \item $\ai \bj \wedge \bk \bl$
        \item $\bi \wedge \aj \bi$
        \item $\bi \wedge \aj \bj$
        \item $\bi \wedge \aj \bk$
        \item $\bi \wedge \bi \bj$
        \item $\bi \wedge \bj$
        \item $\bi \wedge \bj \bk$
        \item $\bi \bj \wedge \bi \bk$
        \item $\bi \bj \wedge \bk \bl$
    \end{enumerate}
    \end{multicols}

Consider the commuting separating disk twists $T_x$ and $T_y$ in Figure \ref{figure:SeparatingCurvesAbelianCycle}.

\begin{figure}[H]
    \centering
    \includegraphics[scale=.5]{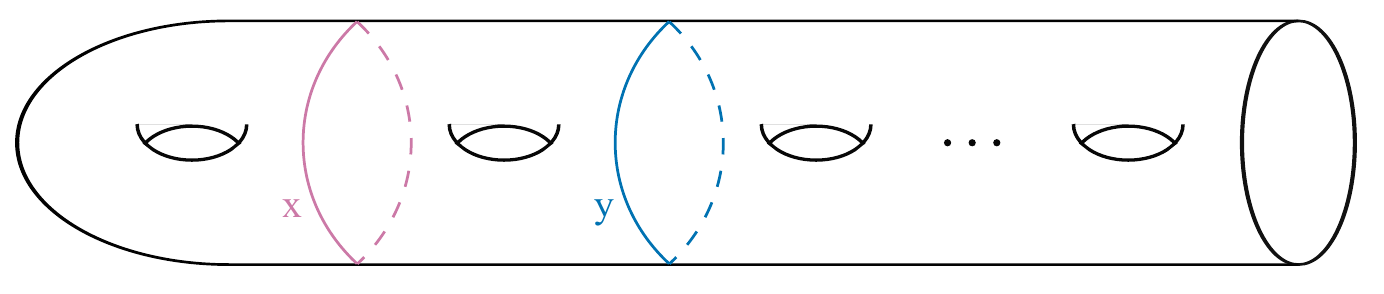}
    \caption{Commuting disk twists $T_x$ and $T_y$}
    \label{figure:SeparatingCurvesAbelianCycle}
\end{figure}

The element $$\sigma_*(\{T_x, T_y \}) = \bar{a}_1 \bar{b}_1 \wedge \bar{a}_2 \bar{b}_2$$ is Type 1. We now find the rest of the basis element types in $\operatorname{im}(\sigma_*)$, building on previous work as we go. We have \begin{gather*} 
X_{jk}(\underbrace{\ai \bi \wedge \aj \bj}_{\text{Type 1}}) = \ai \bi \wedge \aj \bj + \underbrace{\boxed{\ai \bi \wedge \aj \bk}}_{\text{Type 2}} \\
Y_{j \ell}(\underbrace{\ai \bi \wedge \al \bk}_{\text{Type 2}}) = \ai \bi \wedge \al \bk + \underbrace{\boxed{\ai \bi \wedge \bj \bk}}_{\text{Type 3}} \\
X_{ij}(\underbrace{\aj \bj \wedge \ai \bk}_{\text{Type 2}}) = \aj \bj \wedge \ai \bk + \underbrace{\boxed{\ai \bj \wedge \ai \bk}}_{\text{Type 4}} \\
X_{ij}(\underbrace{\ai \bi \wedge \ak \bj}_{\text{Type 2}}) = \ai \bi \wedge \ak \bj + \underbrace{\boxed{\ai \bj \wedge \ak \bj}}_{\text{Type 5}} \\
X_{ij}(\underbrace{\ai \bi \wedge \ak \bl}_{\text{Type 2}}) = \ai \bi \wedge \ak \bl + \underbrace{\boxed{\ai \bj \wedge \ak \bl}}_{\text{Type 6}} \\
Y_{j \ell}(\underbrace{\ai \bj \wedge \al \bk}_{\text{Type 6}}) = \ai \bj \wedge \al \bk + \underbrace{\boxed{\ai \bj \wedge \bj \bk}}_{\text{Type 7}} \\
X_{ij}(\underbrace{\ai \bi \wedge \bk \bl}_{\text{Type 3}}) = \ai \bi \wedge \bk \bl + \underbrace{\boxed{\ai \bj \wedge \bk \bl}}_{\text{Type 8}} \\
Y_{kk}(\underbrace{\ak \bi \wedge \aj \bi}_{\text{Type 5}}) = \ak \bi \wedge \aj \bi + \underbrace{\bk \bi \wedge \aj \bi}_{\text{Type 7}} + \underbrace{\boxed{\bi \wedge \aj \bi}}_{\text{Type 9}} \\
Y_{kk}(\underbrace{\ak \bi \wedge \aj \bj}_{\text{Type 2}}) = \ak \bi \wedge \aj \bj + \underbrace{\bk \bi \wedge \aj \bj}_{\text{Type 3}} + \underbrace{\boxed{\bi \wedge \aj \bj}}_{\text{Type 10}} \\
X_{jk}(\underbrace{\bi \wedge \aj \bj}_{\text{Type 10}}) = \bi \wedge \aj \bj + \underbrace{\boxed{\bi \wedge \aj \bk}}_{\text{Type 11}} \\
Y_{jj}(\underbrace{\bi \wedge \aj \bi}_{\text{Type 9}}) = \bi \wedge \aj \bi + \underbrace{\boxed{\bi \wedge \bi \bj}}_{\text{Type 12}} \\
Y_{jk}(\underbrace{\bi \wedge \ak \bj}_{\text{Type 11}}) = \bi \wedge \ak \bj + \underbrace{\boxed{\bi \wedge \bj}}_{\text{Type 13}} \\
Y_{jk}(\underbrace{\bi \wedge \ak \bk}_{\text{Type 10}}) = \bi \wedge \ak \bk + \underbrace{\boxed{\bi \wedge \bj \bk}}_{\text{Type 14}} \\
Y_{jj}(\underbrace{\aj \bi \wedge \bi \bk}_{\text{Type 7}}) = \aj \bi \wedge \bi \bk + \underbrace{\boxed{\bi \bj \wedge \bi \bk}}_{\text{Type 15}} + \underbrace{\bi \wedge \bi \bk}_{\text{Type 12}} \\
Y_{ii}(\underbrace{\ai \bj \wedge \bk \bl}_{\text{Type 8}}) = \ai \bj \wedge \bk \bl + \underbrace{\boxed{\bi \bj \wedge \bk \bl}}_{\text{Type 16}} + \underbrace{\bj \wedge \bk \bl}_{\text{Type 14}}. \qedhere
\end{gather*} \end{proof} 

\subsection{Dimension counting} \label{section:dimensionIM2} In the following proposition, we bound from below the dimension of $\operatorname{im}(\sigma_*)$. We denote by\footnote{The $``\leq"$ sign emphasizes that we include basis elements containing linear terms, such as $\bi \wedge \ai \bj$.} $IM^{\leq 2}$ the subspace of $\bigwedge^2 B_2^{b_i}$ generated by index-matched basis elements. 

\begin{proposition} \label{prop:dimageHK}
    For $g \geq 4$, the image of the map $\sigma_*: H_2(\mathcal{HK}_g^1; \F_2) \rightarrow H_2(B_2^{b_i}; \F_2)$ has dimension at least on the order of $g^4$.
\end{proposition}

\begin{proof} In Theorem \ref{theorem:johnsonkernel}, we showed that $\operatorname{im}(\sigma_*)$ contains all basis elements of $\bigwedge^2 B_2^{b_i}$ that are \emph{not} in $IM^{\leq 2}$. It suffices to show that $\operatorname{dim}(\bigwedge^2 B_2^{b_i})$ is on the order of $g^4$, while $\operatorname{dim}(IM^{\leq 2})$ is on the order of $g^3$. First, it follows from Proposition \ref{prop:imageBCJdimension} that the vector space $\bigwedge^2 B_2^{b_i}$ has dimension $${\frac{3g^2 + g}{2} \choose 2} = \frac{9}{8}g^4 + O(g^3).$$ 
Now, a basis for $IM^{\leq 2}$ consists of elements of the form $\ai \bar{x} \wedge \bi \bar{y}$, where there are: \begin{itemize}
    \item $g$ ways to choose $i$; and 
    \item $g$ ways\footnote{Our options are all $\bj$ for $1 \leq j \leq g$.} to choose $\bar{x}$; and 
    \item if $\bar{x}=\bi$, there are $2g-1$ ways\footnote{Since $\ai \bi \wedge \ai \bi = 0$.} to choose $\bar{y}$; and
    \item if $\bar{x} \neq \bi$, there are $2g$ ways to choose $\bar{y}$.
\end{itemize} It follows that $$\operatorname{dim}(IM^{\leq 2}) = g(2g-1) + g(g-1) \cdot 2g,$$ and so \begin{gather*}
    \operatorname{dim}(IM^{\leq 2}) = 2g^3 - g.\qedhere \end{gather*}  \end{proof} 

Theorem \ref{maintheorem:HK} follows from Theorem \ref{theorem:johnsonkernel} and Proposition \ref{prop:dimageHK}.

\section{Cohomology classes detected by $B_3^{b_i}$} \label{section:cupHI}

In this section, we prove Theorem \ref{maintheorem:HI}, which we now recall.

\begin{theorem*}[Theorem \ref{maintheorem:HI}]
    For $g \geq 4$, the image of $$\sigma^*: H^2(B_3^{b_i}; \F_2) \rightarrow H^2(\mathcal{HI}_g^1; \F_2)$$ has dimension at least on the order of $g^6$. 
\end{theorem*}

Theorem \ref{maintheorem:HI} will follow from Theorem \ref{theorem:nonIM3} and Proposition \ref{prop:dimageHI}. Following our strategy from Section \ref{section:HK}, we show that the image of the induced map on homology $$\sigma_*: H_2(\mathcal{HI}_g^1; \F_2) \rightarrow H_2(B_3^{b_i}; \F_2)$$ has the desired dimension. Again, we state our theorem in terms of a notion of ``index-matched" elements. 

\subsection{Index-matched elements} Recall from Section \ref{section:handlebodyBCJ} our basis $\mathcal{B}_3^{b_i}$ for $B_3^{b_i}$. Then a basis for $\bigwedge^2 B_3^{b_i}$ is made up of ${|\mathcal{B}_3^{b_i}| \choose 2}$ elements of the form $u \wedge v$, where $u,v \in \mathcal{B}_3^{b_i}$ are distinct. We denote by $IM^3$ the subspace of $\bigwedge^2 B_3^{b_i}$ generated by basis elements of the form $\ai \bar{x} \bar{y} \wedge \bi \bar{w} \bar{z}$ where $\ai \bar{x} \bar{y}$ and $\bi \bar{w} \bar{z}$ are degree-3 elements of $\mathcal{B}_3^{b_i}$. In contrast to the situation in Section \ref{section:dimensionIM2}, where $IM^{\leq 2}$ is spanned by wedge products of polynomials of varying degrees, we require the degrees of $\ai \bar{x} \bar{y}$ and $\bi \bar{w} \bar{z}$ to be exactly $3$. Thus, $$\ai \bj \bk \wedge \bi \bj \bk \in IM^3,$$ whereas $$\ai \bj \wedge \bi \bj \bk \notin IM^3.$$

\begin{theorem} \label{theorem:nonIM3}
    For $g \geq 4$, the image of the map $\sigma_*: H_2(\mathcal{HI}_g^1; \F_2) \rightarrow H_2(B_3^{b_i}; \F_2)$ contains the subspace spanned by all basis elements of $\bigwedge^2 B_3^{b_i} \subset H_2(B_3^{b_i}; \F_2)$ which are not in $IM^3$.
\end{theorem}

Theorem \ref{theorem:nonIM3} will follow from Theorem \ref{theorem:johnsonkernel} together with a series of propositions. In Propositions \ref{prop:aibi}--\ref{prop:remaining2x2}, we show that every basis element $u \wedge v$ for which $u$ and $v$ do not both have degree 3 lies in $\operatorname{im}(\sigma_*)$. We then complete the proof in Propositions \ref{prop:aiajbi}--\ref{prop:bibjbk} by proving that $\operatorname{im}(\sigma_*)$ contains every basis element $u \wedge v \notin IM^3$ with $\operatorname{deg}(u) = \operatorname{deg}(v) = 3$. Unless otherwise stated, we assume that all indices and linear polynomials $\bar{x}, \bar{y}, \bar{z}$ are distinct. We begin with the following proposition.

\begin{proposition}[Type $\ai \bi \wedge \square \square \square$] \label{prop:aibi}
    When $g \geq 4$, all basis elements of the form $\ai \bi \wedge \square \square \square$ in $\bigwedge^2 B_3^{b_i}$ lie in $\operatorname{im}(\sigma_*)$.
\end{proposition} 

\begin{proof} Up to reindexing, the following are the 14 basis elements of the form $\ai \bi \wedge \square \square \square$ in $\bigwedge^2 B_3^{b_i}$: \begin{multicols}{4}
\begin{enumerate}
    \item $\ai \bi \wedge \ai \aj \bi$
    \item $\ai \bi \wedge \ai \aj \bj$
    \item $\ai \bi \wedge \ai \aj \bk$
    \item $\ai \bi \wedge \ai \bi \bj$
    \item $\ai \bi \wedge \ai \bj \bk$
    \item $\ai \bi \wedge \aj \ak \bi$
    \item $\ai \bi \wedge \aj \ak \bj$
    \item $\ai \bi \wedge \aj \ak \bl$
    \item $\ai \bi \wedge \aj \bi \bj$
    \item $\ai \bi \wedge \aj \bi \bk$
    \item $\ai \bi \wedge \aj \bj \bk$
    \item $\ai \bi \wedge \aj \bk \bl$
    \item $\ai \bi \wedge \bi \bj \bk$
    \item $\ai \bi \wedge \bj \bk \bl$
\end{enumerate}
\end{multicols}

We will show that each of these basis element types lies in $\operatorname{im}(\sigma_*)$. First, we explicitly map to the Type 1, 2, and 7 elements with abelian cycles. Consider the commuting disk twists $T_x$ and $T_z$, and the commuting bounding pair annulus twists $T_y T_{y'}^{-1}$ and $T_w T_{w'}^{-1}$ in Figure \ref{figure:aibiCurves}.

\begin{figure}[htbp]
    \centering
    \includegraphics[scale=.4]{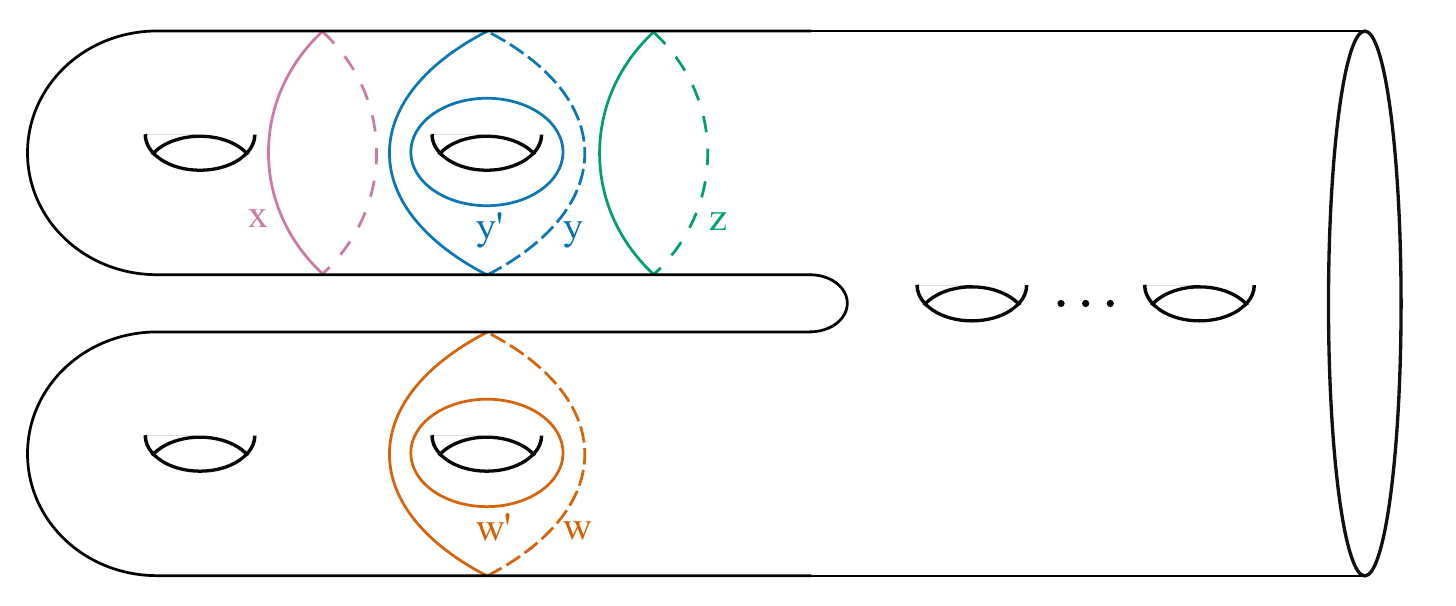}
    \caption{Disk twists and bounding pair annulus twists}
    \label{figure:aibiCurves}
\end{figure}

The element $$\sigma_*(\{ T_x, T_y T_{y'}^{-1} \}) = \bar{a}_1 \bar{b}_1 \wedge \bar{a}_1 \bar{b}_1 \bar{a}_2$$ is Type 1. We also have \begin{gather*} \sigma_*(\{ T_z, T_y T_{y'}^{-1} \}) = \underbrace{\bar{a}_1 \bar{b}_1 \wedge \bar{a}_1 \bar{b}_1 \bar{a}_2}_{\text{Type 1}} + \; \bar{a}_2 \bar{b}_2 \wedge \bar{a}_1 \bar{b}_1 \bar{a}_2 + \underbrace{\bar{a}_2 \bar{b}_2 \wedge \bar{a}_1 \bar{b}_1}_{\text{Theorem \ref{theorem:johnsonkernel}}}. \end{gather*} In particular, $\bar{a}_2 \bar{b}_2 \wedge \bar{a}_1 \bar{b}_1 \bar{a}_2 \in \operatorname{im}(\sigma_*)$ is Type 2. Finally, we have $$\sigma_*(\{ T_x, T_w T_{w'}^{-1} \}) = \bar{a}_1 \bar{b}_1 \wedge \bar{a}_3 \bar{b}_3 \bar{a}_4 + \underbrace{\bar{a}_1 \bar{b}_1 \wedge \bar{a}_3 \bar{b}_3}_{\text{Theorem \ref{theorem:johnsonkernel}}},$$ where $\bar{a}_1 \bar{b}_1 \wedge \bar{a}_3 \bar{b}_3 \bar{a}_4 \in \operatorname{im}(\sigma_*)$ is Type 7. We now show that the remaining element types lie in $\operatorname{im}(\sigma_*)$. We have 
$$X_{jk}(\underbrace{\ai \bi \wedge \ai \aj \bj}_{\text{Type 2}}) = \ai \bi \wedge \ai \aj \bj + \underbrace{\boxed{\ai \bi \wedge \ai \aj \bk}}_{\text{Type 3}}.$$ We also have $$X_{ki}(\underbrace{\ai \bi \wedge \aj \ak \bk}_{\text{Type 7}}) = \ai \bi \wedge \aj \ak \bk + \underbrace{\boxed{\ai \bi \wedge \aj \ak \bi}}_{\text{Type 6}} + \; \green{\ak \bi \wedge \aj \ak \bk + \ak \bi \wedge \aj \ak \bi},$$ where $$X_{ki}(\underbrace{\ak \bk \wedge \aj \ak \bk}_{\text{Type 1}}) = \ak \bk \wedge \aj \ak \bk + \green{\ak \bi \wedge \aj \ak \bk + \ak \bi \wedge \aj \ak \bi} + \underbrace{\ak \bk \wedge \aj \ak \bi}_{\text{Type 3}}.$$ Since $\green{\ak \bi \wedge \aj \ak \bk + \ak \bi \wedge \aj \ak \bi} \in \operatorname{im}(\sigma_*)$, all Type 6 vectors also lie in $\operatorname{im}(\sigma_*)$. We have \begin{gather*}
    Y_{jk}(\underbrace{\ai \bi \wedge \ai \ak \bi}_{\text{Type 1}}) = \ai \bi \wedge \ai \ak \bi + \underbrace{\boxed{\ai \bi \wedge \ai \bi \bj}}_{\text{Type 4}} \\
    Y_{j \ell}(\underbrace{\ai \bi \wedge \ai \al \bk}_{\text{Type 3}}) = \ai \bi \wedge \ai \al \bk + \underbrace{\boxed{\ai \bi \wedge \ai \bj \bk}}_{\text{Type 5}} \\
    X_{j \ell}(\underbrace{\ai \bi \wedge \aj \ak \bj}_{\text{Type 7}}) = \ai \bi \wedge \aj \ak \bj + \underbrace{\boxed{\ai \bi \wedge \aj \ak \bl}}_{\text{Type 8}} \\
    Y_{ii}(\underbrace{\ai \bi \wedge \ai \aj \bj}_{\text{Type 2}}) = \ai \bi \wedge \ai \aj \bj + \underbrace{\boxed{\ai \bi \wedge \aj \bi \bj}}_{\text{Type 9}} + \underbrace{\ai \bi \wedge \aj \bj}_{\text{Theorem \ref{theorem:johnsonkernel}}} \\
    Y_{k \ell}(\underbrace{\ai \bi \wedge \aj \al \bi}_{\text{Type 6}}) = \ai \bi \wedge \aj \al \bi + \underbrace{\boxed{\ai \bi \wedge \aj \bi \bk}}_{\text{Type 10}} \\
    Y_{k \ell}(\underbrace{\ai \bi \wedge \aj \al \bj}_{\text{Type 7}}) = \ai \bi \wedge \aj \al \bj + \underbrace{\boxed{\ai \bi \wedge \aj \bj \bk}}_{\text{Type 11}} \\
    Y_{kk}(\underbrace{\ai \bi \wedge \aj \ak \bl}_{\text{Type 8}}) = \ai \bi \wedge \aj \ak \bl + \underbrace{\boxed{\ai \bi \wedge \aj \bk \bl}}_{\text{Type 12}} + \underbrace{\ai \bi \wedge \aj \bl}_{\text{Theorem \ref{theorem:johnsonkernel}}} \\
    Y_{k \ell}(\underbrace{\ai \bi \wedge \al \bi \bj}_{\text{Type 10}}) = \ai \bi \wedge \al \bi \bj + \underbrace{\boxed{\ai \bi \wedge \bi \bj \bk}}_{\text{Type 13}} \\
    Y_{jk}(\underbrace{\ai \bi \wedge \ak \bk \bl}_{\text{Type 11}}) = \ai \bi \wedge \ak \bk \bl + \underbrace{\boxed{\ai \bi \wedge \bj \bk \bl}}_{\text{Type 14}}. \qedhere
\end{gather*} \end{proof}

We have the following proposition.

\begin{proposition}[Type $\ai \bj \wedge \square \square \square$] \label{prop:aibj}
    When $g \geq 4$, all basis elements of the form $\ai \bj \wedge \square \square \square$ in $\bigwedge^2 B_3^{b_i}$ lie in $\operatorname{im}(\sigma_*)$.
\end{proposition}

\begin{proof}
    First, consider elements of the form $\ai \bj \wedge \bar{x} \bar{y} \bar{z}$ for which $\bar{x}, \bar{y}, \bar{z} \notin \{ \bi, \aj \}$. Then, by Proposition \ref{prop:aibi} we have $\ai \bi \wedge \bar{x} \bar{y} \bar{z} \in \operatorname{im}(\sigma_*)$, and $$X_{ij}(\ai \bi \wedge \bar{x} \bar{y} \bar{z}) = \ai \bi \wedge \bar{x} \bar{y} \bar{z} + \boxed{\ai \bj \wedge \bar{x} \bar{y} \bar{z}}.$$ Similarly, if $\bar{x}, \bar{y}, \bar{z} \notin \{ \ai, \bj \}$, we have $$X_{ji}(\aj \bj \wedge \bar{x} \bar{y} \bar{z}) = \aj \bj \wedge \bar{x} \bar{y} \bar{z} + \boxed{\ai \bj \wedge \bar{x} \bar{y} \bar{z}}.$$ In either of these cases, we say $\ai \bj \wedge \bar{x} \bar{y} \bar{z}$ is Type 0. It remains to show $\ai \bj \wedge \bar{x} \bar{y} \bar{z} \in \operatorname{im}(\sigma_*)$ when the set $\{ \bar{x}, \bar{y}, \bar{z} \}$ nontrivially intersects both $\{\bi, \aj \}$ and $\{ \ai, \bj \}$. Up to reindexing, the following are the 11 basis elements of this form:
    \begin{multicols}{4}
    \begin{enumerate}
        \item $\ai \bj \wedge \ai \aj \bi$
        \item $\ai \bj \wedge \ai \aj \bj$
        \item $\ai \bj \wedge \ai \aj \bk$
        \item $\ai \bj \wedge \ai \ak \bi$
        \item $\ai \bj \wedge \ai \bi \bj$
        \item $\ai \bj \wedge \ai \bi \bk$
        \item $\ai \bj \wedge \aj \ak \bj$
        \item $\ai \bj \wedge \aj \bi \bj$
        \item $\ai \bj \wedge \aj \bj \bk$
        \item $\ai \bj \wedge \ak \bi \bj$
        \item $\ai \bj \wedge \bi \bj \bk$
    \end{enumerate}
    \end{multicols}

We show that basis elements of each of these types lie in $\operatorname{im}(\sigma_*)$. \begin{gather*}
    X_{ij}(\underbrace{\ai \bi \wedge \ai \aj \bk}_{\text{Prop \ref{prop:aibi}}}) = \ai \bi \wedge \ai \aj \bk + \underbrace{\boxed{\ai \bj \wedge \ai \aj \bk}}_{\text{Type 3}} \\ 
    X_{ik}(\underbrace{\ai \bj \wedge \aj \ak \bi}_{\text{Type 0}}) = \ai \bj \wedge \aj \ak \bi + \underbrace{\boxed{\ai \bj \wedge \ai \aj \bi}}_{\text{Type 1}} + \underbrace{\ai \bj \wedge \ai \aj \bk}_{\text{Type 3}} + \underbrace{\ai \bj \wedge \aj \ak \bk}_{\text{Type 0}} \\
    X_{kj}(\underbrace{\ai \bj \wedge \ai \aj \bk}_{\text{Type 3}}) = \ai \bj \wedge \ai \aj \bk + \underbrace{\boxed{\ai \bj \wedge \ai \aj \bj}}_{\text{Type 2}} + \underbrace{\ai \bj \wedge \ai \ak \bj + \ai \bj \wedge \ai \ak \bk}_{\text{Type 0}} \\
    X_{kj}(\underbrace{\ai \bj \wedge \ai \aj \bi}_{\text{Type 1}}) = \ai \bj \wedge \ai \aj \bi + \underbrace{\boxed{\ai \bj \wedge \ai \ak \bi}}_{\text{Type 4}} \\
    X_{ij}(\underbrace{\ai \bi \wedge \ai \bi \bj}_{\text{Prop \ref{prop:aibi}}}) = \ai \bi \wedge \ai \bi \bj + \underbrace{\boxed{\ai \bj \wedge \ai \bi \bj}}_{\text{Type 5}} \\
    Y_{k \ell}(\underbrace{\ai \bj \wedge \ai \al \bi}_{\text{Type 4}}) = \ai \bj \wedge \ai \al \bi + \underbrace{\boxed{\ai \bj \wedge \ai \bi \bk}}_{\text{Type 6}} \\
    X_{ij}(\underbrace{\ai \bi \wedge \aj \ak \bj}_{\text{Prop \ref{prop:aibi}}}) = \ai \bi \wedge \aj \ak \bj + \underbrace{\boxed{\ai \bj \wedge \aj \ak \bj}}_{\text{Type 7}} + \underbrace{\ai \bj \wedge \ai \ak \bj}_{\text{Type 0}} + \underbrace{\ai \bi \wedge \ai \ak \bj}_{\text{Prop \ref{prop:aibi}}} \\
    X_{ij}(\underbrace{\aj \bj \wedge \aj \bi \bj}_{\text{Prop \ref{prop:aibi}}}) = \aj \bj \wedge \aj \bi \bj + \underbrace{\boxed{\ai \bj \wedge \aj \bi \bj}}_{\text{Type 8}} + \underbrace{\ai \bj \wedge \ai \bi \bj}_{\text{Type 5}} + \underbrace{\aj \bj \wedge \ai \bi \bj}_{\text{Prop \ref{prop:aibi}}} \\
    Y_{k \ell}(\underbrace{\ai \bj \wedge \aj \al \bj}_{\text{Type 7}}) = \ai \bj \wedge \aj \al \bj + \underbrace{\boxed{\ai \bj \wedge \aj \bj \bk}}_{\text{Type 9}} \\
    Y_{j \ell}(\underbrace{\ai \bj \wedge \ak \al \bi}_{\text{Type 0}}) = \ai \bj \wedge \ak \al \bi + \underbrace{\boxed{\ai \bj \wedge \ak \bi \bj}}_{\text{Type 10}} \\
    Y_{k \ell}(\underbrace{\ai \bj \wedge \al \bj \bk}_{\text{Type 0}}) = \ai \bj \wedge \al \bj \bk + \underbrace{\boxed{\ai \bj \wedge \bi \bj \bk}}_{\text{Type 11}}. \qedhere
\end{gather*} \end{proof}

We have the following proposition.

\begin{proposition}[Type $\bi \bj \wedge \square \square \square$] \label{prop:bibj}
    When $g \geq 4$, all basis elements of the form $\bi \bj \wedge \square \square \square$ in $\bigwedge^2 B_3^{b_i}$ lie in $\operatorname{im}(\sigma_*)$.
\end{proposition}

\begin{proof}
    We will show $\bi \bj \wedge \bar{x} \bar{y} \bar{z} \in \operatorname{im}(\sigma_*)$. Choose $1 \leq k \leq g$ such that $\ak \notin \{ \bar{x}, \bar{y}, \bar{z} \}$. Then, if $\ai \notin \{ \bar{x}, \bar{y}, \bar{z} \}$, we have \begin{equation} \label{eq:bibj}
        Y_{ik}(\underbrace{\ak \bj \wedge \bar{x} \bar{y} \bar{z}}_{\text{Prop \ref{prop:aibj}}}) = \ak \bj \wedge \bar{x} \bar{y} \bar{z} + \boxed{\bi \bj \wedge \bar{x} \bar{y} \bar{z}}. 
    \end{equation} Otherwise, we have \begin{gather*}
        Y_{ik}(\underbrace{\ak \bj \wedge \ai \bar{x} \bar{y}}_{\text{Prop \ref{prop:aibj}}}) = \ak \bj \wedge \ai \bar{x} \bar{y} + \underbrace{\ak \bj \wedge \bk \bar{x} \bar{y}}_{\text{Prop \ref{prop:aibj}}} + \; \boxed{\bi \bj \wedge \ai \bar{x} \bar{y}} + \bi \bj \wedge \bk \bar{x} \bar{y},
    \end{gather*} where $\bi \bj \wedge \bk \bar{x} \bar{y} \in \operatorname{im}(\sigma_*)$ by Equation \eqref{eq:bibj}. \end{proof}

Propositions \ref{prop:aibi}, \ref{prop:aibj}, and \ref{prop:bibj} give us the following corollary.

\begin{corollary}[Type $\square \square {\wedge \; } {\square \square \square}$] \label{corollary:2wedge3}
    When $g \geq 4$, all basis elements of $\bigwedge^2 B_3^{b_i}$ of the form $\square \square {\wedge \; } {\square \square \square}$ lie in $\operatorname{im}(\sigma_*)$.
\end{corollary}

We have the following proposition.

\begin{proposition}[Type $\bi \wedge \square \square \square$] \label{prop:bi}
    When $g \geq 4$, all basis elements of the form $\bi \wedge \square \square \square$ in $\bigwedge^2 B_3^{b_i}$ lie in $\operatorname{im}(\sigma_*)$.
\end{proposition}

\begin{proof}
    We will show $\bi \wedge \bar{x} \bar{y} \bar{z} \in \operatorname{im}(\sigma_*)$. Choose $1 \leq j \leq g$ such that $\aj \notin \{ \bar{x}, \bar{y}, \bar{z} \}$. Then we have \begin{gather*}
        Y_{jj}(\underbrace{\aj \bi \wedge \bar{x} \bar{y} \bar{z}}_{\text{Prop \ref{prop:aibj}}}) = \aj \bi \wedge \bar{x} \bar{y} \bar{z} + \underbrace{\bi \bj \wedge \bar{x} \bar{y} \bar{z}}_{\text{Prop \ref{prop:bibj}}} + \boxed{\bi \wedge \bar{x} \bar{y} \bar{z}}. \qedhere
    \end{gather*} \end{proof}

We have the following proposition.

\begin{proposition}[$\bigwedge^2 B_2^{b_i}$]  \label{prop:remaining2x2}
    When $g \geq 4$, the subspace $\bigwedge^2 B_2^{b_i} \subset \bigwedge^2 B_3^{b_i}$ lies in $\operatorname{im}(\sigma_*)$.
\end{proposition}

\begin{proof}
In Theorem \ref{theorem:johnsonkernel}, we showed that the subspace spanned by all basis elements of $\bigwedge^2 B_2^{b_i}$ which are not index-matched is in $\operatorname{im}(\sigma_*)$. Up to reindexing, the following 10 index-matched elements are the remaining basis elements of $\bigwedge^2 B_2^{b_i}$ that we must show lie in $\operatorname{im}(\sigma_*)$: \begin{multicols}{4}
    \begin{enumerate}
    \item $\ai \bi \wedge \ai \bj$
    \item $\ai \bi \wedge \aj \bi$
    \item $\ai \bi \wedge \bi \bj$
    \item $\ai \bj \wedge \aj \bi$
    \item $\ai \bj \wedge \aj \bk$
    \item $\ai \bj \wedge \ak \bi$
    \item $\ai \bj \wedge \bi \bj$
    \item $\ai \bj \wedge \bi \bk$
    \item $\bi \wedge \ai \bi$
    \item $\bi \wedge \ai \bj$
\end{enumerate}
\end{multicols}
Note that the basis elements of Types 1--8 are (possibly after reindexing) of the form $\ai \bar{x} \wedge \bi \bar{y}$. Given such an element, choose $1 \leq k \leq g$ for which $k \neq i$ and $\ak \notin \{ \bar{x}, \bar{y} \}$. Then we have \begin{gather*} Y_{kk}(\underbrace{\ai \bar{x} \wedge \ak \bi \bar{y}}_{\text{Corollary \ref{corollary:2wedge3}}}) = \ai \bar{x} \wedge \ak \bi \bar{y} + \underbrace{\ai \bar{x} \wedge \bi \bk \bar{y}}_{\text{Corollary \ref{corollary:2wedge3}}} + \; \boxed{\ai \bar{x} \wedge \bi \bar{y}}. \end{gather*} It follows that all elements of the Types 1--8 lie in $\operatorname{im}(\sigma_*)$. Similarly, we have \begin{gather*}
    Y_{kk}(\underbrace{\bi \wedge \ai \ak \bar{x}}_{\text{Prop \ref{prop:bi}}}) = \bi \wedge \ai \ak \bar{x} + \underbrace{\bi \wedge \ai \bk \bar{x}}_{\text{Prop \ref{prop:bi}}} + \; \boxed{\bi \wedge \ai \bar{x}}
\end{gather*} when $\bar{x} = \bi$ (Type 9) or $\bar{x} = \bj$ (Type 10). \end{proof}

We have now shown that all terms of the form $u \wedge v$ with $u \in B_2^{b_i}$ and $v \in B_3^{b_i}$ lie in $\operatorname{im}(\sigma_*)$. From now on, we call such elements \emph{lower order terms (LOTs)}. It remains to address wedge products of two degree-3 polynomials in $B_3^{b_i}$.

\begin{proposition}[Type $\ai \aj \bi \wedge \square \square \square$] \label{prop:aiajbi}
    When $g \geq 4$, all non-index-matched basis elements of the form $\ai \aj \bi \wedge \square \square \square$ in $\bigwedge^2 B_3^{b_i}$ lie in $\operatorname{im}(\sigma_*)$.
\end{proposition}

\begin{proof}

We will show $\ai \aj \bi \wedge \bar{x} \bar{y} \bar{z} \in \operatorname{im}(\sigma_*)$ when $\bar{x}, \bar{y}, \bar{z} \notin \{\bi, \bj, \ai \}$. Up to reindexing, the following are the eight basis elements of this form:
\begin{multicols}{3}
\begin{enumerate} 
    \item $\ai \aj \bi \wedge \aj \ak \bk$
    \item $\ai \aj \bi \wedge \aj \ak \bl$
    \item $\ai \aj \bi \wedge \aj \bk \bl$ 
    \item $\ai \aj \bi \wedge \ak \al \bk$
    \item $\ai \aj \bi \wedge \ak \al \bm$
    \item $\ai \aj \bi \wedge \ak \bk \bl$  
    \item $\ai \aj \bi \wedge \ak \bl \bm$ 
    \item $\ai \aj \bi \wedge \bk \bl \bm$ 
    \end{enumerate}
    \end{multicols} Consider the commuting bounding pair annulus twists $T_x T_{x'}^{-1}$ and $T_y T_{y'}^{-1}$ as in Figure \ref{figure:3wedge3boundingpairs}.

\begin{remark}
Note that Types 5, 7, and 8 only appear when $g \geq 5$. When $g = 4$, no basis elements require 5 distinct indices, and our calculations for these carefully use at most 4 distinct indices.
\end{remark}

\begin{figure}[htbp]
    \centering
    \includegraphics[scale=.4]{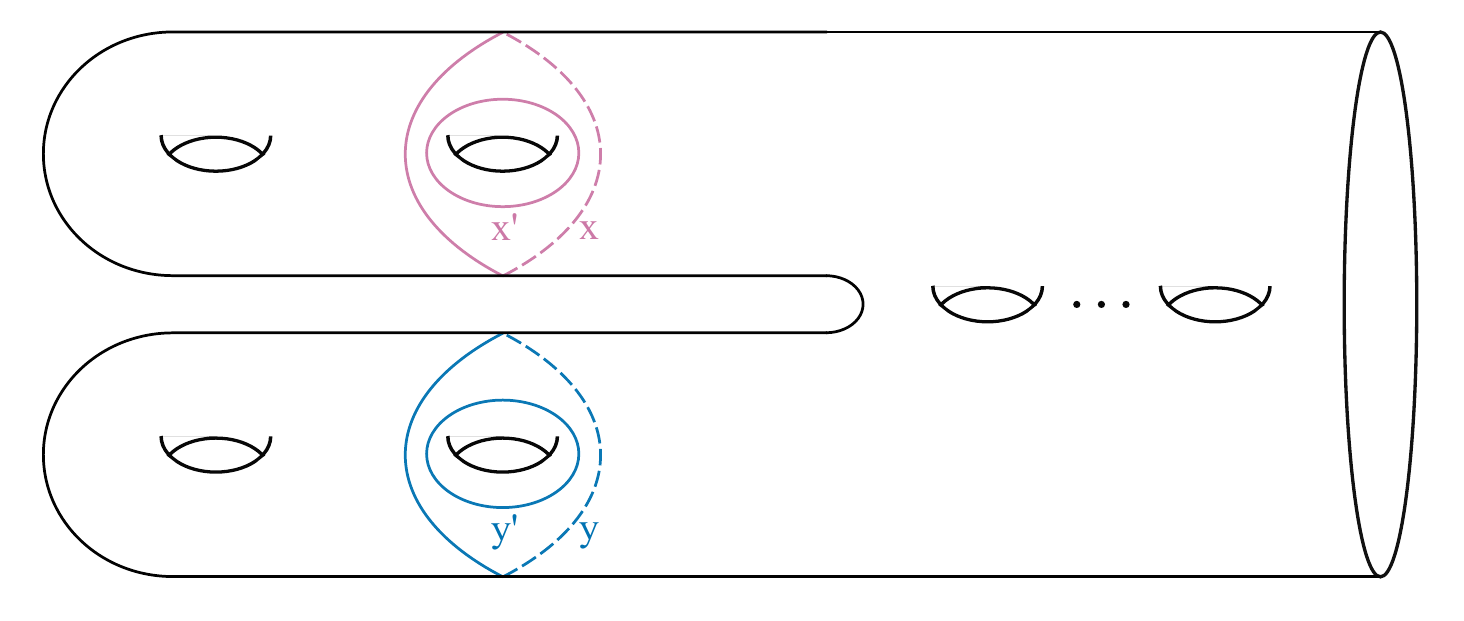}
    \caption{Bounding pair annulus twists $T_x T_{x'}^{-1}$ and $T_y T_{y'}^{-1}$}
    \label{figure:3wedge3boundingpairs}
\end{figure} \noindent For some choice of indexing, we have $$\sigma_*(\{T_x T_{x'}^{-1}, T_y T_{y'}^{-1}\}) = \underbrace{\boxed{\ai \aj \bi \wedge \ak \al \bk}}_{\text{Type 4}} + \; \text{LOTs}.$$ Now we will show that the remaining element types lie in $\operatorname{im}(\sigma_*)$. We have \begin{gather*}
X_{j \ell}(\underbrace{\ai \aj \bi \wedge \ak \al \bk}_{\text{Type 4}}) = \ai \aj \bi \wedge \ak \al \bk + \underbrace{\boxed{\ai \aj \bi \wedge \aj \ak \bk}}_{\text{Type 1}} \\
X_{k \ell}(\underbrace{\ai \aj \bi \wedge \aj \ak \bk}_{\text{Type 1}}) = \ai \aj \bi \wedge \aj \ak \bk + \underbrace{\boxed{\ai \aj \bi \wedge \aj \ak \bl}}_{\text{Type 2}} \\
Y_{kk}(\underbrace{\ai \aj \bi \wedge \aj \ak \bl}_{\text{Type 2}}) = \ai \aj \bi \wedge \aj \ak \bl + \underbrace{\boxed{\ai \aj \bi \wedge \aj \bk \bl}}_{\text{Type 3}} + \underbrace{\ai \aj \bi \wedge \aj \bl}_{\text{LOT}} \\
X_{km}(\underbrace{\ai \aj \bi \wedge \ak \al \bk}_{\text{Type 4}}) = \ai \aj \bi \wedge \ak \al \bk + \underbrace{\boxed{\ai \aj \bi \wedge \ak \al \bm}}_{\text{Type 5}} \\
Y_{\ell \ell}(\underbrace{\ai \aj \bi \wedge \ak \al \bk}_{\text{Type 4}}) = \ai \aj \bi \wedge \ak \al \bk + \underbrace{\boxed{\ai \aj \bi \wedge \ak \bk \bl}}_{\text{Type 6}} + \underbrace{\ai \aj \bi \wedge \ak \bk}_{\text{LOT}} \\
Y_{\ell m}(\underbrace{\ai \aj \bi \wedge \ak \am \bm}_{\text{Type 4}}) = \ai \aj \bi \wedge \ak \am \bm + \underbrace{\boxed{\ai \aj \bi \wedge \ak \bl \bm}}_{\text{Type 7}} \\
Y_{kk}(\underbrace{\ai \aj \bi \wedge \ak \bl \bm}_{\text{Type 7}}) = \ai \aj \bi \wedge \ak \bl \bm + \underbrace{\boxed{\ai \aj \bi \wedge \bk \bl \bm}}_{\text{Type 8}} + \underbrace{\ai \aj \bi \wedge \bl \bm}_{\text{LOT}}. \qedhere
\end{gather*} \end{proof}

We have the following proposition.

\begin{proposition}[Type $\ai \aj \bk \wedge \square \square \square$] \label{prop:aiajbk}
    When $g \geq 4$, all non-index-matched basis elements of the form $\ai \aj \bk \wedge \square \square \square$ in $\bigwedge^2 B_3^{b_i}$ lie in $\operatorname{im}(\sigma_*)$.
\end{proposition}

\begin{proof} We will show $\ai \aj \bk \wedge \bar{x} \bar{y} \bar{z} \in \operatorname{im}(\sigma_*)$ when $\bar{x}, \bar{y}, \bar{z} \notin \{ \bi, \bj, \ak \}$. If $\ai \notin \{ \bar{x}, \bar{y}, \bar{z} \}$, then $\ai \aj \bi \wedge \bar{x} \bar{y} \bar{z}$ is not index-matched and lies in $\operatorname{im}(\sigma_*)$ by Proposition \ref{prop:aiajbi}. In this case, we have \begin{gather*}
    X_{ik}(\ai \aj \bi \wedge \bar{x} \bar{y} \bar{z}) = \ai \aj \bi \wedge \bar{x} \bar{y} \bar{z} + \boxed{\ai \aj \bk \wedge \bar{x} \bar{y} \bar{z}}.
\end{gather*} Similarly, if $\aj \notin \{ \bar{x}, \bar{y}, \bar{z} \}$, then Proposition \ref{prop:aiajbi} tells us $\ai \aj \bj \wedge \bar{x} \bar{y} \bar{z} \in \operatorname{im}(\sigma_*)$ and we have \begin{gather*}
    X_{jk}(\ai \aj \bj \wedge \bar{x} \bar{y} \bar{z}) = \ai \aj \bj \wedge \bar{x} \bar{y} \bar{z} + \boxed{\ai \aj \bk \wedge \bar{x} \bar{y} \bar{z}}. \end{gather*} (Note that, in both cases, we used the assumption $\bar{x}, \bar{y}, \bar{z} \notin \{ \bi, \bj, \ak \}$.) It remains to show $\ai \aj \bk \wedge \ai \aj \bl \in \operatorname{im}(\sigma_*)$. By the above discussion, we know $\ai \aj \bk \wedge \ai \al \bl \in \operatorname{im}(\sigma_*)$, and we have \begin{gather*}
    X_{j \ell}(\ai \aj \bk \wedge \ai \al \bl) = \ai \aj \bk \wedge \ai \al \bl + \boxed{\ai \aj \bk \wedge \ai \aj \bl}. \qedhere
\end{gather*} \end{proof}

We have the following proposition.

\begin{proposition}[Type $\ai \bi \bj \wedge \square \square \square$] \label{prop:aibibj}
    When $g \geq 4$, all non-index-matched basis elements of the form $\ai \bi \bj \wedge \square \square \square$ in $\bigwedge^2 B_3^{b_i}$ lie in $\operatorname{im}(\sigma_*)$.
\end{proposition}

\begin{proof} We will show $\ai \bi \bj \wedge \bar{x} \bar{y} \bar{z} \in \operatorname{im}(\sigma_*)$ when $\bar{x}, \bar{y}, \bar{z} \notin \{ \bi, \ai, \aj \}$. If $\bj \notin \{ \bar{x}, \bar{y}, \bar{z} \}$, then $\ai \aj \bi \wedge \bar{x} \bar{y} \bar{z}$ is not index-matched and lies in $\operatorname{im}(\sigma_*)$ by Proposition \ref{prop:aiajbi}. In this case, we have \begin{gather*}
    Y_{jj}(\ai \aj \bi \wedge \bar{x} \bar{y} \bar{z}) = \ai \aj \bi \wedge \bar{x} \bar{y} \bar{z} + \boxed{\ai \bi \bj \wedge \bar{x} \bar{y} \bar{z}} + \underbrace{\ai \bi \wedge \bar{x} \bar{y} \bar{z}}_{\text{LOT}}.
\end{gather*} It remains to show $\ai \bi \bj \wedge \bj \bar{x} \bar{y} \in \operatorname{im}(\sigma_*)$. First, we show $\ai \bi \bj \wedge \bj \bk \bl \in \operatorname{im}(\sigma_*)$. By the above discussion, we have $\ai \bi \bj \wedge \ak \bk \bl \in \operatorname{im}(\sigma_*)$ and \begin{gather*} Y_{jk}(\ai \bi \bj \wedge \ak \bk \bl) = \ai \bi \bj \wedge \ak \bk \bl + \boxed{\ai \bi \bj \wedge \bj \bk \bl}. \end{gather*} Now we assume $\ai \bi \bj \wedge \bj \bar{x} \bar{y}$ cannot be reindexed to $\ai \bi \bj \wedge \bj \bk \bl$. Then we can find $1 \leq m \leq g$ such that $m \notin \{i,j \}$ and (since $g \geq 4$) at least one of $\am$ and $\bm$ is not in $\{ \bar{x}, \bar{y} \}$. If $\am \notin \{ \bar{x}, \bar{y} \}$, we have $\ai \bi \bj \wedge \am \bar{x} \bar{y} \in \operatorname{im}(\sigma_*)$ and \begin{gather*}
Y_{jm}(\ai \bi \bj \wedge \am \bar{x} \bar{y}) = \ai \bi \bj \wedge \am \bar{x} \bar{y} + \boxed{\ai \bi \bj \wedge \bj \bar{x} \bar{y}}.\end{gather*} If $\bm \notin \{ \bar{x}, \bar{y} \}$, we have $\ai \bi \bj \wedge \bm \bar{x} \bar{y} \in \operatorname{im}(\sigma_*)$ and \begin{gather*}
    X_{jm}(\ai \bi \bj \wedge \bm \bar{x} \bar{y}) = \ai \bi \bj \wedge \bm \bar{x} \bar{y} + \boxed{\ai \bi \bj \wedge \bj \bar{x} \bar{y}}. \qedhere
\end{gather*} \end{proof}

We have the following proposition.

\begin{proposition}[Type $\ai \bj \bk \wedge \square \square \square$] \label{prop:aibjbk}
    When $g \geq 4$, all non-index-matched basis elements of the form $\ai \bj \bk \wedge \square \square \square$ in $\bigwedge^2 B_3^{b_i}$ lie in $\operatorname{im}(\sigma_*)$. 
\end{proposition}

\begin{proof} We will show $\ai \bj \bk \wedge \bar{x} \bar{y} \bar{z} \in \operatorname{im}(\sigma_*)$ when $\bar{x}, \bar{y}, \bar{z} \notin \{ \bi, \aj, \ak \}$. If $\bj \notin \{ \bar{x}, \bar{y}, \bar{z} \}$, then $\ai \aj \bk \wedge \bar{x} \bar{y} \bar{z} \in \operatorname{im}(\sigma_*)$ by Proposition \ref{prop:aiajbk} and we have \begin{gather*}
    Y_{jj}(\ai \aj \bk \wedge \bar{x} \bar{y} \bar{z}) = \ai \aj \bk \wedge \bar{x} \bar{y} \bar{z} + \boxed{\ai \bj \bk \wedge \bar{x} \bar{y} \bar{z}} + \underbrace{\ai \bk \wedge \bar{x} \bar{y} \bar{z}}_{\text{LOT}}.
\end{gather*} Similarly, if $\bk \notin \{ \bar{x}, \bar{y}, \bar{z} \}$,  then $\ai \ak \bj \wedge \bar{x} \bar{y} \bar{z} \in \operatorname{im}(\sigma_*)$ by Proposition \ref{prop:aiajbk} and we have \begin{gather*}
    Y_{kk}(\ai \ak \bj \wedge \bar{x} \bar{y} \bar{z}) = \ai \ak \bj \wedge \bar{x} \bar{y} \bar{z} + \boxed{\ai \bj \bk \wedge \bar{x} \bar{y} \bar{z}} + \underbrace{\ai \bj \wedge \bar{x} \bar{y} \bar{z}}_{\text{LOT}}.
\end{gather*} It remains to show $\ai \bj \bk \wedge \bj \bk \bar{x} \in \operatorname{im}(\sigma_*)$. This element vanishes if $\bar{x} = \ai$, so assume otherwise. We have $\ai \bi \bk \wedge \bj \bk \bar{x} \in \operatorname{im}(\sigma_*)$ by Proposition \ref{prop:aibibj} and \begin{gather*}
    X_{ij}(\ai \bi \bk \wedge \bj \bk \bar{x}) = \ai \bi \bk \wedge \bj \bk \bar{x} + \boxed{\ai \bj \bk \wedge \bj \bk \bar{x}}. \qedhere \end{gather*} \end{proof}

We have the following proposition.

\begin{proposition}[Type $\bi \bj \bk \wedge \square \square \square$] \label{prop:bibjbk}
    When $g \geq 4$, all non-index-matched basis elements of the form $\bi \bj \bk \wedge \square \square \square$ in $\bigwedge^2 B_3^{b_i}$ lie in $\operatorname{im}(\sigma_*)$.
\end{proposition}

\begin{proof}
    We will show $\bi \bj \bk \wedge \bar{x} \bar{y} \bar{z} \in \operatorname{im}(\sigma_*)$ when $\bar{x}, \bar{y}, \bar{z} \notin \{ \ai, \aj, \ak \}$. If any of $\bar{x}, \bar{y}, \bar{z}$ is an $``\bar{a}"$ term, then $\bi \bj \bk \wedge \bar{x} \bar{y} \bar{z} \in \operatorname{im}(\sigma_*)$ by Propositions \ref{prop:aiajbi}--\ref{prop:aibjbk}. It remains to show $\bi \bj \bk \wedge \bl \bm \bn \in \operatorname{im}(\sigma_*)$, where we stray from our usual convention and allow for indices to not be distinct. Up to reindexing, we can assume $\ell \notin \{i,j,k \}$. In this case, we have $\bi \bj \bk \wedge \al \bm \bn \in \operatorname{im}(\sigma_*)$ and \begin{gather*}
        Y_{\ell \ell}(\bi \bj \bk \wedge \al \bm \bn) = \bi \bj \bk \wedge \al \bm \bn + \boxed{\bi \bj \bk \wedge \bl \bm \bn} + \underbrace{\bi \bj \bk \wedge \bm \bn}_{\text{LOT}}. \qedhere \end{gather*} \end{proof}

Theorem \ref{theorem:nonIM3} follows from Propositions \ref{prop:aibi}--\ref{prop:bibjbk}.

\subsection{Dimension counting} \label{section:dimensionIM3} In the following proposition, we bound from below the dimension of $\operatorname{im}(\sigma_*)$.

\begin{proposition} \label{prop:dimageHI}
     For $g \geq 4$, the image of the map $\sigma_*: H_2(\mathcal{HI}_g^1; \F_2) \rightarrow H_2(B_3^{b_i}; \F_2)$ has dimension at least on the order of $g^6$.
\end{proposition}

\begin{proof} In Theorem \ref{theorem:nonIM3}, we showed that $\operatorname{im}(\sigma_*)$ contains all basis elements of $\bigwedge^2 B_3^{b_i}$ that do not lie in $IM^3$. It suffices to show that $\operatorname{dim}(\bigwedge^2 B_3^{b_i})$ is on the order of $g^6$, while $\operatorname{dim}(IM^3)$ is on the order of $g^5$. First, it follows from Proposition \ref{prop:imageBCJdimension} that $\bigwedge^2 B_3^{b_i}$ has dimension $${\frac{7g^3 + 5g}{6} \choose 2} = \frac{49}{72} g^6 + O(g^5).$$ A basis for $IM^3$ consists of elements of the form $$\ai \bar{x} \bar{y} \wedge \bi \bar{w} \bar{z}$$ for which the sets $\{ \ai, \bar{x}, \bar{y} \}$ and $\{ \bi, \bar{w}, \bar{z} \}$ each consist of three distinct elements. We determine $\operatorname{dim}(IM^3)$. First, there are $g$ ways to choose $i$. Next, we must choose $\bar{x}, \bar{y}, \bar{w}, \bar{z}$, and we split this up into two cases. 

Consider the case $\bar{x} = \bi$. This leaves $2g-2$ ways to choose $\bar{y}$. We cannot have $\{ \bar{w}, \bar{z} \} = \{ \ai, \bar{y} \}$, but any other choice of $\{ \bar{w}, \bar{z} \}$ such that $\bar{w}$ and $\bar{z}$ and $\bi$ are distinct will work. This leaves\footnote{We account for $\blue{ \bi \notin \{ \bar{w}, \bar{z} \} }$ and remove $\orange{\{\ai, \bar{y}\}}$ as a choice.} $${2g \; \blue{-1} \choose 2} \; \orange{-1}$$ ways to choose $\{ \bar{w}, \bar{z} \}$.

Now consider the case $\bi \notin \{ \bar{x}, \bar{y} \}$. At least one of $\bar{x}, \bar{y}$ must be a ``$\bar{b}$" term, so set $\bar{x} = \bj$ for $1 \leq i \neq j \leq g$. There are $g-1$ ways to do this. There are $2g-3$ ways to choose $\bar{y} \notin \{ \ai, \bi, \bj \}$. Now $\{ \bar{w}, \bar{z} \}$ can be anything as long as $\bar{w}$ and $\bar{z}$ and $\bi$ are distinct, and there are ${2g-1 \choose 2}$ ways to do this. Putting this all together, we have $$\operatorname{dim}(IM^3) = g \left[ (2g-2) \left( \binom{2g-1}{2} - 1 \right) + (g-1)(2g-3) \binom{2g-1}{2} \right],$$ and so \begin{gather*} \operatorname{dim}(IM^3) = 4 g^5 + O(g^4). \qedhere \end{gather*} \end{proof}


Theorem \ref{maintheorem:HI} follows from Theorem \ref{theorem:nonIM3} and Proposition \ref{prop:dimageHI}.

\end{document}